\pdfoutput=1
\documentclass[smallextended,envcountsect]{svjour3}
\smartqed
\journalname{}

\usepackage[utf8]{inputenc}
\usepackage{amsmath,mathtools}
\usepackage{amssymb}
\usepackage{algorithm,algorithmic}
\usepackage{anyfontsize}
\usepackage{eqparbox}
\usepackage{enumitem}
\usepackage{euscript}
\usepackage{exscale,relsize}
\usepackage{graphicx}
\usepackage{booktabs}
\usepackage{xcolor}
\usepackage[T1]{fontenc}
\usepackage{textcomp}
\usepackage{charter}
\usepackage{pifont}
\usepackage{url}
\usepackage{tikz}
\usepackage{mathrsfs}
\usepackage{pgfplots}
\usepackage{caption}
\usepackage{color}
\usepackage{wasysym}
\usepackage{float}
\usepackage{fewerfloatpages}
\usepackage{subfig}
\usepackage{wrapfig}
\pgfplotsset{minor tick style={draw=none}}
\usepackage[colorlinks,hyperindex]{hyperref}
\hypersetup{linktocpage=true,citecolor=dblue,linkcolor=dgreen}
\usepackage[square,sort,numbers]{natbib}
\usepackage{lipsum}
\usepackage{lineno}

\definecolor{labelkey}{rgb}{0,0.08,0.45}
\definecolor{refkey}{rgb}{0,0.6,0.0}
\definecolor{Brown}{rgb}{0.45,0.0,0.05}
\definecolor{dgreen}{rgb}{0.00,0.49,0.00}
\definecolor{dblue}{rgb}{0,0.08,0.75}
\definecolor{ffwwqq}{rgb}{1.,0.4,0.}
\definecolor{qqzzqq}{rgb}{0.,0.6,0.}
\definecolor{qqqqff}{rgb}{0.,0.,1.}
\definecolor{dred}{HTML}{D90404}
\definecolor{orng}{HTML}{D35400}
\definecolor{cb-black}      {RGB}{  0,   0,   0}
\definecolor{cb-blue-green} {RGB}{  0,  073,  073}
\definecolor{cb-green-sea}  {RGB}{  0, 146, 146}
\definecolor{cb-rose}       {RGB}{255, 109, 182}
\definecolor{cb-salmon-pink}{RGB}{255, 182, 119}
\definecolor{cb-purple}     {RGB}{ 73,   0, 146}
\definecolor{cb-blue}       {RGB}{ 0, 109, 219}
\definecolor{cb-lilac}      {RGB}{182, 109, 255}
\definecolor{cb-blue-sky}   {RGB}{109, 182, 255}
\definecolor{cb-blue-light} {RGB}{182, 219, 255}
\definecolor{cb-burgundy}   {RGB}{146,   0,   0}
\definecolor{cb-brown}      {RGB}{146,  73,   0}
\definecolor{cb-clay}       {RGB}{219, 209,   0}
\definecolor{cb-green-lime} {RGB}{ 36, 255,  36}
\definecolor{cb-yellow}     {RGB}{255, 255, 109}
\definecolor{bred}{HTML}{FF0000}
\definecolor{bpurp}{HTML}{BF00BF}
\definecolor{bblu}{HTML}{0000FF}
\definecolor{bcyan}{HTML}{00BFBF}
\definecolor{byellow}{HTML}{BFBF00}
\definecolor{bgreen}{HTML}{008000}

\newcommand{\bibfilename}{block-arxiv3}

\newtheorem{assumption}[theorem]{Assumption}

\newtheorem{fact}[theorem]{Fact}

\renewcommand\theenumi{(\roman{enumi})}
\renewcommand\theenumii{(\alph{enumii})}

\DeclareMathOperator*{\Argmin}{Argmin}

\newcommand{\scal}[2]{{\langle{{#1}\mid{#2}}\rangle}}
\DeclareMathOperator*{\minimize}{minimize}
\newcommand{\HH}{\mathcal H}
\newcommand{\HHH}{\boldsymbol{\mathcal H}}
\newcommand{\CCX}{\bigtimes_{i\in I}C_i}

\newcommand{\bz}{\boldsymbol{z}}
\newcommand{\bv}{\boldsymbol{v}}
\newcommand{\by}{\boldsymbol{y}}

\newcommand{\bx}{\boldsymbol{x}}
\newcommand{\bg}{\boldsymbol{g}}

\newcommand{\bc}{\boldsymbol{c}}

\newcommand{\NN}{\ensuremath{\mathbb N}}
\newcommand{\RR}{\ensuremath{\mathbb{R}}}

\newcommand{\menge}[2]{\big\{{#1}~\big |~{#2}\big\}}

\DeclareMathOperator{\lmo}{LMO}

\modulolinenumbers[5]

\def\dateconvex{12-11T04-48}
\def\datenonconvex{12-11T04-48}
\def\dateadaptive{12-11T04-52}

\begin{document}
\title{Flexible block-iterative analysis for the Frank-Wolfe
algorithm}
\author{Gábor Braun \and Jannis Halbey \and Sebastian Pokutta \and Zev Woodstock}

\institute{Gábor Braun, Jannis Halbey, Sebastian Pokutta, Zev Woodstock
  \at
  Zuse Institute Berlin and
  Institute of Mathematics, Technische Universit\"at Berlin, Germany
  \and
  Zev Woodstock (\texttt{woodstzc@jmu.edu}), Corresponding author \at
  James Madison University, Harrisonburg, VA, USA;\\ The work of ZW
was supported by
the National Science Foundation under grant DMS-2532423.}

\maketitle
\begin{abstract}
We prove that the block-coordinate Frank-Wolfe (BCFW) algorithm
converges with state-of-the-art rates in both convex and nonconvex
settings under a very mild ``block-iterative'' assumption.
This appears to be the first
result on BCFW addressing the setting of nonconvex objective
functions with Lipschitz-continuous gradients and no additional
assumptions. This analysis newly allows for (I) progress without
activating the most-expensive
linear minimization oracle(s), LMO(s), at every iteration, (II)
parallelized updates that do not
require all LMOs, and therefore (III) deterministic parallel update
strategies that take into account the numerical cost of the
problem's LMOs. Our results
apply for short-step BCFW as well as an
adaptive method for convex functions. New relationships between
updated coordinates and primal progress are proven, and 
a favorable speedup is demonstrated using
\texttt{FrankWolfe.jl}.
\end{abstract}

\keywords{
  conditional gradient,
  block-iterative algorithm,
Frank-Wolfe,
projection-free first-order method,
parallel updates
}

\subclass{
  49M27, % decomposition methods
  49M37, % numerical methods based on nonlinear programming
  65K05, % Numerical mathematical programming methods
  90C26, % nonconvex programming
  90C30 % nonlinear programming
}

\section{Introduction}
\label{sec:1}

Given a smooth function \(f\) that maps from a finite Cartesian
product of $m$ real Hilbert spaces
$\HHH\coloneqq\bigoplus_{i=1}^{m}\HH_i$ to $\mathbb{R}$
and a product of nonempty compact
convex subsets $\bigtimes_{i=1}^mC_i\subset\HHH$
with \(C_{i} \subseteq \HH_{i}\),
we seek to solve
the following problem
\begin{equation}
\label{e:p}
\minimize_{\bx\in C_1\times\ldots\times C_m} f(\bx),
\end{equation}
which has applications in matrix factorization, support vector
machine training, sequence labeling, intersection verification, 
and more
\cite{Bomz24,Brau22,Ding08,Laco13,Osok16,Sreb04,Wang16,Wood24}.
Frank-Wolfe (F-W), also known as \emph{conditional gradient},
methods have become an
increasingly popular choice for solving \eqref{e:p} on large-scale
problems, because their method of enforcing set constraints, namely
the \emph{linear minimization oracle}, is oftentimes computationally
faster than other techniques such as projection algorithms
\cite{Comb21LMO}. A \emph{linear minimization oracle} $\lmo_C$
for a compact convex set \(C \subseteq \HHH\),
computes for any linear objective
$\bc\in\HHH$, a point in $\Argmin_{\bx\in C}\scal{\bc}{\bx}$. 
The oracle approach
is advantageous for problems such as the maximum matching problem 
\cite{rothvoss14,braun2014matching}, where efficient linear minimization is possible despite 
large linear program formulations; this also makes reduction between problems easier
\cite{Braun2015LP-SDP}.

Although \eqref{e:p} can be solved via the classical Frank-Wolfe
algorithm, it would necessitate that, at every iteration,
the linear minimization oracle for
$C_1\times\ldots\times C_m$ is evaluated.
This step can cause a computational bottleneck, since
\begin{equation}
\lmo_{C_1\times\ldots\times C_m}(\bx^1,\ldots,\bx^m)
=(\lmo_{C_1}\bx^1,\ldots,\lmo_{C_m}\bx^m),
\end{equation}
i.e., evaluating the Cartesian LMO amounts to computing $m$ separate
LMOs. To avoid this slowdown, there has been an increasing
effort over the last decade to reduce the per-iteration complexity
required by classical Frank-Wolfe algorithms while maintaining
theoretical guarantees of convergence
\cite{Beck15,Brau22,Bomz24,Laco13,Osok16,Wang16}. These
benefits make performing a single iteration on larger-scale
problems more tractable, and oftentimes allow for the more
efficient use of kilowatts in practice.

Here we are interested in improvements making use of the product
structure of the feasible region, which can later be combined with other
improvement techniques to better use linear minimization,
such as delaying updates via local acceleration
\cite{diakonikolas2020locally},
generalized self-concordant objective functions
\cite{carderera2021simple},
and \emph{boosting}, i.e., using multiple linear minimizations
to choose a direction for progress
\cite{combettes20boostfw}.

Perhaps the earliest work using the product structure of the feasible
region was
\cite{Patr98}, which proved that, for Armijo and exact line
searches, asymptotic convergence to a solution of \eqref{e:p} could
be achieved by, at each iteration, only updating one component
(also called coordinate) of
the iterate and thereby requiring one LMO evaluation. In
particular, \cite{Patr98} showed that convergence is guaranteed as
long as an \emph{essentially cyclic} selection scheme is used, that
is, as long as there exists some $K$ such that all $m$
components are updated at least once
over each consecutive $K$
iterations. In other words, the index $\mathfrak{i}(t)$ of
the component updated at iteration $t\in\NN$, satisfies
\begin{equation}
\label{e:ecyclic}
\{\mathfrak{i}(t),\ldots,\mathfrak{i}(t+K-1)\}=\{1,\ldots,m\}.
\end{equation}
About 17 years later, \cite{Beck15} significantly
improved upon these results for the \emph{cyclic} setting ($K=m$),
by (A)
widening to a scope of many more Frank-Wolfe variants (e.g., adaptive
steps, open-loop predefined steps, and backtracking) and (B)
deriving modern convergence rates. This cyclic scheme has shown to
be particularly useful with
randomly shuffling the order of updating the components
for each cycle.
In contrast to the deterministic methods, 
\cite{Laco13} showed that by selecting uniformly at random the
component to update in each iteration,
one can also solve \eqref{e:p}. Since then, two methods
have been proposed that select one component to update based on a
suboptimality criterion: the one in \cite{Osok16} is stochastic and
selects the component via a non-uniform distribution, while the
Gauss-Southwell, or ``greedy'', update scheme of \cite{Bomz24} is
deterministic. Such techniques can provide improved per-iteration
progress, although they are agnostic to the numerical costs of the
selected LMO.

In contrast to singleton-update schemes, the vanilla Frank-Wolfe
algorithm and several of its modern variants \cite{Wang16,Bomz24}
are particularly suitable for updating several components of an
iterate in parallel, which can yield better per-iteration progress.
In these \emph{block-iterative} settings, at iteration $t\in\NN$,
a \emph{block}
$I_t\subset\{1,\ldots,m\}$ of components are updated (possibly in
parallel) while leaving
the remaining components in $\{1,\ldots,m\}\setminus I_t$
unchanged. Updated components $i\in I_t$ are modified via a
Frank-Wolfe subroutine which relies on evaluating $\lmo_{C_i}$, and
therefore the selection of $I_t$ has a
great influence on the per-iteration cost of the algorithm.
Some of the earliest results for parallel block updated
Frank-Wolfe algorithms again arise from \cite{Patr98}, which proved
convergence (without rates) for parallel synchronous updates with
the full updating scheme $I_t=\{1,\ldots,m\}$
and a uniform step size across all components\footnote{Although \cite[Section~4]{Patr98} also contains results for
more general selection schemes of $I_t$, they do not apply to the
Frank-Wolfe setting (see \cite[Table~1]{Patr98}).}.
As pointed out by \cite{Beck15}, using a single step size in all
components can impede progress, since the relative
scale between componentwise constraints can be significant.
The recent work \cite{Bomz24} allowed for full updates with
variable componentwise step sizes, also significantly improving
convergence rates in certain settings.

However, outside of full-updates, it appears that only
\cite{Wang16} allows for block-sizes larger than $1$.
In particular, for a fixed block-size $p$, the results in 
\cite{Wang16} permit selecting the updated coordinates
uniformly at random. This application
is ideal when all LMOs are expected to require the same amount of
time, and $p$ processor cores are available. However, unless all
the operators $(\lmo_{C_i})_{i\in\{1,\ldots,m\}}$ require
similar levels of computational effort, there appear to be no other
good options for leveraging parallelism. In particular,
regardless of the step sizes considered, it appears that (prior to
this work) no block selection technique for a F-W algorithm allows
block sizes to change between iterations, and
there are no \emph{deterministic} rules which even allow for blocks
$I_t$ with sizes between $1$ and $m$. This poses a
significant drawback from a computational perspective, because the
current ``state-of-the-art'' leaves very little customizability or
adaptability in how the block-updates are selected. In particular,
a central goal of this work is to allow for the design of
cost-aware update techniques which take into account the relative
numerical cost of the LMOs of a given problem, and utilize all
available processors at a given iteration.
Newly-allowed update methods for BCFW will also be compared to
existing update techniques, including the 
essentially-cyclic singleton-update scheme \eqref{e:ecyclic} that
also allows for the design of cost-aware update strategies without
parallel block updates.

Even though the Frank-Wolfe algorithm predates many methods
which rely on proximity operators, advances in block-coordinate
proximal algorithms seem to have outpaced those in the Frank-Wolfe
literature. So, in this article we consider parallel and partial
componentwise updates for BCFW under the following
assumption, which comes from the proximal-based literature
\cite{Otta88}.

\begin{assumption}
\label{a:1}
There exists a positive integer $K$ such that, for every
iteration $t$,
\begin{equation}
\label{e:blockA}
(\forall i\in\{1,\ldots,m\}) \quad i\in\bigcup_{k=t}^{t+K-1}I_k.
\end{equation}
\end{assumption}
\noindent We emphasize the flexibility of Assumption~\ref{a:1}: in addition
to allowing for the computation of expensive LMOs at any (bounded)
rate, Assumption~\ref{a:1} allows deterministic parallelized
block-updates of variable sizes, up to the user.
Assumption~\ref{a:1} also unifies several existing 
selection schemes. Below are some example use-cases. 
\begin{enumerate}
\item
\label{flex1}
With the $I_{t}$ singletons, this becomes the \emph{
essentially cyclic} selection scheme \eqref{e:ecyclic} of
\cite{Tsen01,Patr98}; if additionally $K=m$, Assumption~\ref{a:1}
becomes the \emph{cyclic} scheme of \cite{Beck15}.
\item
With $I_{t}=\{1,\ldots,m\}$, this becomes the
\emph{full} selection method, also called \emph{parallel}
\cite{Patr98,Bomz24}.
\item
If $p$ processor cores are available, one can queue $p$ many LMO
operations to be performed in parallel, hence
satisfying Assumption~\ref{a:1} with $K=\lceil m/p \rceil$. 
This strategy is well-suited for reducing processor wait
times if the LMOs for the selected blocks require roughly the same amount of computational time (which
occurs, e.g., in \cite{Beck15,Osok16}).
\item
\label{flex4}
If the operators $(\lmo_{C_i})_{i\in \{1,\ldots,m\}}$ require
drastically different
levels of computational
time (e.g., where some LMOs are fast, while others require
comparatively slower computations such as eigendecomposing
a large matrix or solving a large linear program), one can postpone
evaluating the most expensive LMOs, provided they are evaluated
once every $K$ iterations. The experiments in Section~\ref{sec:num}
demonstrate that, by repeatedly iterating on the ``cheaper''
components, one can nonetheless provide good per-iteration progress
on the overall problem.\footnote{This strategy is reminiscent of
the Shamanskii-type Newton/Chord algorithms that only perform
numerically expensive Hessian updates once over a finite sequence of
iterations \cite{Kell95,Sham67}.}
\item
\label{flexn}
 Assumption~\ref{a:1} also allows for a quasi-stochastic strategy:
For all iterations from $t$ to $t+K-2$, use any stochastic
selection technique; then, at iteration $t+K-1$, additionally
activate the (potentially empty) set of components which were not
selected by the stochastic method.
\end{enumerate}

\subsection*{Contributions}
Our main contributions are threefold. To the best of our knowledge,
this article contains the first result concerning converge of BCFW
in the nonconvex case where the objective function has a
Lipschitz-continuous gradient and no extra assumptions.
We are only aware of one work which addresses BCFW with nonconvex objectives,
namely \cite{Bomz24} establishes
linear convergence under several assumptions including a
Kurdyka-Łojasiewicz-type inequality.
As is
standard in Frank-Wolfe
methods, Theorem~\ref{t:ncv-conv} proves that after $t$ iterations,
the algorithm is guaranteed to produce a point with
\emph{Frank-Wolfe gap} (a quantity closely related to
stationarity \cite{CGSurvey}) being at most
$\mathcal{O}(1/\sqrt{t})$. Second, for the case
of convex objective functions,
an $\mathcal{O}(1/t)$ primal gap convergence rate is proven for an
adaptive step size version of BCFW which does not require a-priori
smoothness estimation (Theorem~\ref{t:flex}); in consequence,
Corollary~\ref{c:sc} establishes convergence for
short-step BCFW with a rate and constant which matches short-step
FW \cite[Theorem~2.2]{CGSurvey}.
Third, throughout the entire article we only assume the flexible
block-activation scheme, Assumption~\ref{a:1},
which unifies many previous activation schemes for BCFW into one
simple framework and allows for
new block-selection strategies,
e.g., those available for some prox-based algorithms.
On toy problems for which there is a significantly disparate cost
of linear minimization oracles, these new selection strategies are
shown to perform comparably, or even \emph{better} than existing
methods in iterations, gradient evaluations, LMO
evaluations, and time. 

The remainder of the article is organized as follows.
Section~\ref{sec:bg} details background and preliminary results;
Section~\ref{sec:steps} presents a general formulation of BCFW, a
discussion on step size variants, and the common progress
estimation for convex objective functions.
Section~\ref{sec:adaptive-step-size} considers convex objective
functions and proves convergence under both adaptive step sizes
and short-step sizes.
Section~\ref{sec:ss} proves a convergence guarantee for nonconvex
objective functions with Lipschitz-continuous gradients.
Finally, Section~\ref{sec:num} shows computational experiments.

\subsection{Notation, standing assumptions, and auxiliary results}
\label{sec:bg}

Let \(I \coloneqq \{1, 2, \dots, m\}\),
and we consider the direct sum $\HHH\coloneqq\bigoplus_{i\in I}\HH_i$
of real Hilbert spaces $\HH_i$.
We denote points of \(\HHH\) by bold letters,
and components in the direct sum by upper indices,
i.e., $\bx=(\bx^1,\bx^2,\ldots,\bx^m)\in\HHH$
with \(\bx^{i} \in \HH_{i}\).
The inner product on $\HHH$ is
$\scal{\bx}{\by}_{\HHH}=\sum_{i\in I}\scal{\bx^i}{\by^i}_{\HH_i}$,
yielding the norm identity $\|\bx-\by\|_{\HHH}^2=\sum_{i\in
I}\|\bx^i-\by^i\|_{\HH_i}^2$.
For notational convenience, we treat the \(\HH_{i}\) as orthogonal subspaces of \(\HHH\),
in particular, \(\bx = \sum_{i \in I} \bx^{i}\).
We will omit the subscripts \(\HHH\), \(\HH_{i}\)
from norms and scalar products; this will not cause ambiguity as
all are restrictions of the ones on \(\HHH\).
For \(J \subseteq I\), let \(\bx^{J} \coloneqq \sum_{i \in J}
\bx^{j}\) be the part of \(\bx\) in the components \(\HH_{i}\)
for \(i \in J\).
For $i\in I$, let $C_i$ be a nonempty compact convex subset
of \(\HH_{i}\).
For $J\subset I$, let \(\bigtimes_{i \in J} C_{i}\) be the set of
points \(\bx \in \HHH\) with \(\bx^{i} \in C_{i}\) for all \(i\in J\)
and \(\bx^{i} = 0\) for \(i \notin J\).
Let \(D_{J}\) be the diameter of
\(\bigtimes_{i \in J} C_{i}\) (treated as a subset of
$\bigoplus_{i\in J}\HH_i \subset \HHH$).
We shall use the simplified notation
$D_i\coloneqq D_{\{i\}}$ and $D\coloneqq D_{I}$.

Let $f$ be a Fréchet differentiable function mapping from
\(\bigtimes_{i \in J} C_{i}\) to
$\mathbb{R}$.
We denote partial gradients by
$\nabla^Jf(\bx)\coloneqq (\nabla f(\bx))^{J}$.
For $L_f>0$, a function $f$ is \emph{$L_f$-smooth} on
a convex set $C$ if
\begin{equation}
\label{e:2}
(\forall \bx,\by \in C)\quad f(\by)-f(\bx)\leq\scal{\nabla f(\bx)}{\by-\bx}+
\dfrac{L_f}{2}\|\by-\bx\|^2;
\end{equation}
which holds, e.g., if $\nabla f$ is $L_f$-Lipschitz continuous
\cite{CGSurvey}.
Recall that $f$ is \emph{convex} on a convex set \(C\) if
\begin{equation}
\label{e:1}
\left(\forall \bx,\by \in C\right) \quad
\scal{\nabla f(\bx)}{\by-\bx}\leq
f(\by)-f(\bx).
\end{equation}
For nonempty \(J \subset I\) and
$\bx^J\in\bigoplus_{i\in
J}\HH_i$, the \emph{linear minimization oracle} $\lmo_{J}(\bx^J)$
returns a point in \(\Argmin_{\bv\in \bigtimes_{i \in J}
C_i}\scal{\bx^{J}}{\bv}\); also set $\lmo_i\coloneqq \lmo_{\{i\}}$.
A \emph{partial Frank-Wolfe gap} is given by
\begin{equation}
\label{e:Jgap}
(\forall J\subset I)\bigg(\forall \bx\in\CCX\bigg)\quad
  G_J(\bx) =
  \scal{\nabla^J f(\bx)}{\bx^{J}
    - \lmo_{J}(\nabla^J f(\bx))}
  ,
\end{equation}
with $G_i=G_{\{i\}}$. The \emph{Frank-Wolfe gap} (F-W
gap) of $f$ over $\CCX$ at
$\bx\in\HHH$ is given by
\begin{equation}
\label{e:fwgap}
G_{I}(\bx)\coloneqq \sup_{\bv\in \CCX}\scal{\nabla
f(\bx)}{\bx-\bv}=\sum_{i\in I} G_i(\bx).
\end{equation}
Note that, for every $\bx\in\CCX$ and every $J\subset I$, we have
$G_J(\bx)\geq 0$. The F-W gap vanishes at a solution of \eqref{e:p}
in the
following sense \cite{CGSurvey}
\begin{equation}
\label{e:opt}
\bx\text{ is a stationary point of
} \minimize_{\bx\in \CCX} f(\bx)
\quad\Leftrightarrow\quad
\begin{cases}
\bx\in\CCX\\
G_{I}(\bx)\leq 0.
\end{cases}
\end{equation}
Before proceeding further, we
gather several useful results.

\begin{lemma}
\label{lem:pcb}
Let $(C_i)_{i\in I}$ be nonempty compact convex subsets of real
Hilbert spaces $(\HH_i)_{i\in I}$, let $f\colon\CCX\to\mathbb{R}$
be convex, let $J\subset I$ be nonempty, and let $G_J$ be
given by \eqref{e:Jgap}. Then,
\begin{equation}
\label{e:26}
\Big(\forall \bz\in\CCX\Big)\quad 
G_{J}(\bz)\geq f(\bz) - \min_{\substack{\bx\in\bigtimes_{i\in
I}C_i\\ \bx^{I\setminus J}=\bz^{I\setminus J}}} f(\bx)\geq 0.
\end{equation}
\end{lemma}
\begin{proof}
Let $\bx_{\bz}^*\in\Argmin_{\bx^J\in\bigtimes_{i\in
J}C_i} f(\bx^J + \bz^{I\setminus J})$.
By \eqref{e:Jgap} and optimality of the LMO,
$G_J(\bz)=\scal{\nabla^J f(\bz)}{\bz^J-\lmo_J(\nabla^J
f(\bz))}
\geq\scal{\nabla^J f(\bz)}{\bz^J-\bx_{\bz}^*}
=\scal{\nabla f(\bz)}{\bz-(\bx_{\bz}^*+\bz^{I\setminus
J})}$,
so using convexity yields \eqref{e:26}.
\end{proof}

We will use the perspective function \(\rho\)
of a Huber loss,
to simplify handling the
minimum in the short-step formula \eqref{e:shortA} below.
\begin{equation}
\label{e:ph}
\rho\colon\RR\times\RR_{>0}\to\RR\colon(x,b)\mapsto
    \begin{cases}
      |x| - \frac{b}{2} & \text{if } |x| \geq b\\
      \frac{|x|^{2}}{2 b} & \text{if } |x| \leq b
      .
    \end{cases}
\end{equation}
\begin{fact}
\label{f:rho}
The function \(\rho\) is proper and jointly convex
\cite[Proposition~8.25,
Ex.~8.44]{Livre1}. Also, for $b>0$ and $x\geq 0$, note
$\rho(x,\cdot)$ is clearly decreasing on $\mathbb{R}_{>0}$;
$\rho(\cdot,b)$ is even and hence increasing on
$\mathbb{R}_{\geq 0}$ \cite[Proposition~11.7]{Livre1}; and
$x-\frac{b}{2}\leq \rho(x,b)$ as a tangent line of
the convex function $\rho(\cdot,b)$. Furthermore, $\rho$ is
subadditive \cite[Example~10.5]{Livre1}:
\begin{equation}
  \label{eq:1}
  \sum_{i=1}^{n} \rho(x_{i}, b_{i})
  \geq
  \rho
  \left(
    \sum_{i=1}^{n} x_{i}, \sum_{i=1}^{n} b_{i}
  \right)
  .
\end{equation}
\end{fact}

\begin{lemma}
  \label{lem:max-quad}
  Let \(x\), \(c\) be nonnegative numbers and let 
  \(b\) be a positive number. Then,
  \begin{equation}
    \label{eq:max-quad}
    \frac{(x - c)^{2}}{2 b} + c
    \geq
\rho(x,b)
  \end{equation}
\end{lemma}
  \begin{proof}
Fixing \(x\) and \(b\),
the left-hand side of \eqref{eq:max-quad} is a quadratic function of
\(c\) with minimum attained at
\(c = x - b\) for \(x \geq b\),
and \(c = 0\) for \(x \leq b\).
Thus,
\begin{equation}
\quad\text{if \(x \geq b\), then}\quad
  \frac{(x - c)^{2}}{2 b} + c
  \geq
  x - \frac{b}{2}
  ;\text{ if \(x \leq b\), then}\quad
  \frac{(x - c)^{2}}{2 b} + c
  \geq
  \frac{x^{2}}{2 b}
  .
\end{equation}
\end{proof}
The following takes inspiration from \cite{Beck13}
and includes a nonmonotone sequence $(a_t)_{t\in\NN}$ representing
extra progress.
\begin{lemma}
\label{l:recursion}
Let $h_t$ and $a_t$ be nonnegative numbers for \(t \in \NN\),
let $b>0$, let $\rho$ be given by \eqref{e:ph}, and
suppose that
$h_t-h_{t+1}\geq\rho(h_t+a_t,b)$
for every $t\in\NN$.
Then $(h_t)_{t\in\NN}$ decreases monotonically and 
\begin{equation}
(\forall t\in\NN)\quad
h_t\leq\begin{cases}
\frac{b}{2}-a_0&\text{if\;\;}t=1\\
\dfrac{2b}{t-1+
\frac{2b}{h_{1}}+
\sum_{k=1}^{t-1}\frac{2a_k}{h_{1}}
+\left(\frac{a_k}{h_{1}}\right)^2
}
&\text{if\;\;}t\geq 2.
\end{cases}
\end{equation}
\end{lemma}
\begin{proof}
Since
$h_t-h_{t+1}\geq
\rho(h_{t}+a_{t},b)\geq 0$, the sequence $(h_t)_{t\in\NN}$ is
decreasing. Since $x-\frac{b}{2}\leq\rho(x,b)$, our
recursion yields $h_0+a_0-b/2\leq h_0-h_1$, and rearranging proves
\(h_{1} \leq b/2 - a_{0}\).
Next, we observe that since $(h_t)_{t\in\NN}$ is monotonic and
$\rho$ is strictly monotonically increasing in its first argument,
for every $k\geq 1$,
we have $\rho(b,b)=\frac{b}{2}\geq h_1\geq h_k\geq h_k-h_{k-1}\geq
\rho(h_k+a_k,b)$, so $h_k+a_k\leq b$ and hence
\(\rho(h_{k} + a_{k}, b) = (h_{k} + a_{k})^{2} / (2b)\).
Now, fix $t\in\NN\setminus\{0\}$. 
If $h_{t+1}=0$, we are done; otherwise, by monotonicity we have
$0<h_{t+1}\leq \cdots\leq h_{1}$. So, 
\begin{equation}
\begin{split}
(\forall k\in\{1,\ldots,t\})\quad
\frac{1}{h_{k+1}}-\frac{1}{h_k}&
=\frac{h_k-h_{k+1}}{h_kh_{k+1}}
\geq
\frac{(h_k+a_k)^2}{2bh_kh_{k+1}}
\\
&=\frac{1}{2b}\left(\frac{h_k}{h_{k+1}}+\frac{2 a_k}{h_{k+1}}+\frac{a_k^2}{h_kh_{k+1}}\right)\\
\label{e:b}
&\geq
\frac{1}{2b}\left(1+\frac{2 a_k}{h_{1}}+\left(\frac{a_k}{h_{1}}\right)^2\right).
\end{split}
\end{equation}
We sum \eqref{e:b} over $k\in\{1,\ldots,t\}$ to find
\begin{equation}
\label{e:b2}
\frac{1}{h_{t+1}}-\frac{1}{h_{1}}
\geq
\frac{1}{2b}\left(
t+
\sum_{k=1}^t\frac{2 a_k}{h_{1}}
+\left(\frac{a_k}{h_{1}}\right)^2\right),
\end{equation}
and rearranging \eqref{e:b2} completes the result.
\end{proof}

\subsection{Generic form of BCFW}
\label{sec:steps}
Consider the generic form of the block-coordinate
Frank-Wolfe algorithm shown in Algorithm~\ref{a:bcgg}.
The selection strategies of the blocks
$(I_t)_{t\in\NN}$ in
\cite{Beck15,Laco13,Patr98} arise as special cases.
\begin{algorithm}[ht]
\caption{Block-Coordinate Frank-Wolfe (BCFW), Generic form}
\label{a:bcgg}  
\begin{algorithmic}[1]
\REQUIRE Function $f\colon\CCX\to\mathbb{R}$, gradient $\nabla f$,
point $\bx_0\in\CCX$, linear minimization oracles $(\lmo_i)_{i\in
I}$
\FOR{$t=0, 1$ \textbf{to} $\dotsc$}
\STATE Choose a nonempty block $I_t\subset I$
\STATE $\bg_t\leftarrow \nabla f(\bx_t)$
\FOR{$i=1$ \textbf{to} $m$}
\IF{$i\in I_t$}
\label{l:blockg}
\STATE $\bv_{t}^i\leftarrow\lmo_i(\bg_t^i)$
\STATE\label{l:stepg}
$\gamma_t^i
 \leftarrow $ Step size parameter (see also
Sections~\ref{sec:adaptive-step-size}, \ref{sec:ss})
\STATE \label{l:updateg} $\bx_{t+1}^i \leftarrow
\bx_{t}^i+\gamma_t^i(\bv_t^i-\bx_t^i)$
\ELSE
\STATE \label{l:lazyg}
$\bx_{t+1}^i\leftarrow \bx_{t}^i$
\ENDIF
\ENDFOR
\ENDFOR
\end{algorithmic}
\end{algorithm}
\begin{remark}
\label{r:short-motiv}
For $L_f$-smooth objective functions $f$, 
the smoothness inequality \eqref{e:2} and
Line~\ref{l:updateg} of Algorithm~\ref{a:bcgg}
yield
\begin{equation}
\label{e:short-motiv}
f(\bx_{t+1})-f(\bx_t)
\leq\sum_{i\in I_t}\gamma_t^i\scal{\nabla^i
f(\bx_{t})}{\bv_t^i-\bx_t^i}+
\dfrac{L_f}{2}(\gamma_t^i)^2\|\bv_{t}^i-\bx_t^i\|^2.
\end{equation}
To tighten the bound \eqref{e:short-motiv}, a common step size
choice is to minimize the summands via a componentwise analogue of
the so-called \emph{short-step} \cite{Demy70,Beck15}:
\begin{equation}
\label{e:shortA}
\begin{aligned}
\gamma_t^i &= \Argmin_{\gamma\in\left[0,1\right]}\left(-\gamma
  G_i(\bx_t)+\gamma^2\frac{L_f }{2}
  \|\bv_t^i-\bx_t^i\|^2\right)\\
&=\min\left\{\dfrac{G_i(\bx_t)}{L_f \|\bv_t^i-\bx_t^i\|^2},
  1\right\}.
\end{aligned}
\tag{short}
\end{equation}
This is analyzed in Sections~\ref{sec:convss} and \ref{sec:ss} for
convex and nonconvex objectives respectively.
Section~\ref{sec:flex3} addresses a situation where a similar
update to \eqref{e:shortA} is performed using an estimation.
\end{remark}

This work focuses on using short-step step sizes that rely on $L_f$
(or its estimator in Section~\ref{sec:flex3}). In
contrast, the following
examples demonstrate that Algorithm~\ref{a:bcgg}
need not converge using componentwise analogues of classical F-W
step sizes.
\begin{example}[Non-convergent componentwise line search]
\label{ex:nonconv}
Frank-Wolfe methods at a point $\bx_t$ with vertex $\bv_t$ are
commonly known to converge where step sizes are selected by a
line search, i.e., $\gamma_t=\Argmin_{\gamma\in[0,1]}
f(\bx_t+\gamma(\bv_t-\bx_t))$ \cite{CGSurvey}.
However, when line search step sizes are chosen componentwise, namely
via
\begin{equation}
\label{e:19}
  \gamma_t^i \in \Argmin_{\gamma\in[0,1]}
  f\big(\bx_t+\gamma (\bv_t^i-\bx_t^i)\big),
\end{equation}
Algorithm~\ref{a:bcgg} need not converge.
  Let \(I \coloneqq \{1, 2\}\),
  \(\HH_{1} = \HH_{2} = \RR\),
  \(C_{1} = C_{2}= [-1, 1]\),
  and \(f(\bx) \coloneqq (\bx^1+\bx^2)^{2}\);
  in particular, \(L_{f} = 4\).
  The minimal function value of \(0\) is
  attained at the points \(\bx\)
  for which \(\bx^{1} = - \bx^{2}\). 
  With full block activation $I_t=\{1,2\}$ and componentwise 
  line search \eqref{e:19}, the iterates of
  Algorithm~\ref{a:bcgg} satisfy
  \(\bx_{t+1}^{1} = - \bx_{t}^{2}\)
  and \(\bx_{t+1}^{2} = - \bx_{t}^{1}\).
  Hence, a possible sequence of iterates is
  \(((-1)^{t}, (-1)^{t})\),
  which does not converge to optimality in function value.
\end{example}

\begin{example}[Non-convergent componentwise short-step]
\label{ex:nonconv2}
In singleton-update cyclic schemes, it is possible to use
a variant of \eqref{e:shortA} where, for every $i\in I$, $L_f$ is
replaced by the Lipschitz constant \(\beta_{i}\) of $\nabla f$
over the component $C_i$. More precisely, componentwise
short-steps $\gamma_t^i=\min\{1,
G_i(\bx_t)/\beta_i \|\bv_t^i-\bx_t^i\|^2\}$ allow for larger
step sizes \cite{Beck15}, since $\beta_i\leq L_f$.
However, in Example~\ref{ex:nonconv}, $\beta_1=\beta_{2}=2\neq
4=L_f$, i.e.,
$\gamma_t^i$ is the same as in Example~\ref{ex:nonconv}.
Therefore, using Algorithm~\ref{a:bcgg} with $I_t=\{1,2\}$, this
short-step variant may produce the same iterates as
Example~\ref{ex:nonconv},
which do not get close to the optimal solution.
\end{example}

The following technical lemma is for combining with inequalities of
the form
\(f(\bx_{t}) - f(\bx_{t+1}) \geq \tau P_{t}\) ($\tau> 0$)
to construct convergence results in
Sections~\ref{sec:adaptive-step-size}
and \ref{sec:ss}. For $\tau=1$, the above inequality naturally
arises as a consequence of convexity and smoothness (seen in
Fact~\ref{f:smooth3}). 
The $(M_t)_{t\in\NN}$ play the role of approximate smoothness
constants.
\begin{lemma}
  \label{lem:short-progress}
Let $\CCX\subset\HHH$ be a finite product
of nonempty compact convex sets \(C_{i}\), let $D$ be the diameter of
$\CCX$, let
$f\colon\CCX\to\mathbb{R}$ be Fréchet differentiable,
let $K$ be a positive integer,
let $(M_t)_{t\in\NN}$ be a sequence of positive numbers,
and for every $J\subset I$, let $G_J$ be given by \eqref{e:Jgap}
and set $\bv^J_t=\lmo_J(\bg_t)$. 
In the setting of Algorithm~\ref{a:bcgg},
for every $t\in\NN$ and $J\subset I$, set $\bv^J_t=\lmo_J(\bg_t)$,
set
\begin{equation}
P_t=
   \scal{\bg_{t}}{\bx_{t} - \bx_{t+1}}
    - \frac{M_{t+1} \|\bx_{t} - \bx_{t+1}\|^{2}}{2}
    +
    \frac{\|\bg_{t}^{I \setminus I_{t}}
      - \bg_{t+1}^{I \setminus I_{t}}\|^{2}}{2 M_{t+1}},
\end{equation}
set $A_t=\sum_{k=1}^{K-1} G_{I_{t+k-1}\cap
(I_{t+k}\cup\cdots\cup I_{t+K-1})}(\bx_{t+k})\geq 0$, 
and for every $i\in I_t$ let Line~\ref{l:stepg} be specified by
\begin{equation}
  \gamma_t^i=\min\bigg\{1,\frac{G_i(\bx_t)}{M_{t+1}
    \|\bx_t^i-\bv_t^i\|^2}\bigg\}.
\end{equation}
Then, for every $t\in\NN$, the following hold:
\begin{enumerate}
\item
    \label{l:block-progressb}
$(\forall J \subseteq I\setminus I_{t})\;\;
P_t
    \geq
    \rho\big(|G_{I_{t} \cup J}(\bx_{t}) - \scal{\bg_{t+1}^{J}}{
        \bx_{t}^{J} - \bv_{t}^{J}}|,
      M_{t+1} \|\bx_{t}^{I_{t} \cup J} - \bv_{t}^{I_{t} \cup
J}\|^{2}\big)
$.
\item
    \label{l:block-progressa}
$
\sum_{k=0}^{K-1}P_{t+k}
    \geq
    \rho \left(
      G_{I_{t} \cup \dots \cup I_{t+K-1}}(\bx_{t}) + A_{t},
      \sum_{k=1}^{K} M_{t+k} D^{2}
    \right)
      .$
\end{enumerate}
\end{lemma}
\begin{proof}
Let $i\in I_t$. We claim
\begin{equation}
  \label{e:25}
  \scal{\bg_{t}^{i}}{\bx_{t}^{i} - \bx_{t+1}^{i}}
  - \frac{M_{t+1} \|\bx_{t}^{i} - \bx_{t+1}^{i}\|^{2}}{2}
  =
  \rho(\scal{\bg_{t}^{i}}{\bx_{t}^{i} - \bv_{t}^{i}},
  M_{t+1} \|\bx_{t}^{i} - \bv_{t}^{i}\|^{2})
  .
\end{equation}
We distinguish two cases depending on \(\gamma_{t}^i\).
If
\(\scal{\bg_{t}^{i}}{\bx_{t}^{i}-\bv_{t}^{i}} \geq
M_{t+1}\|\bx_{t}^{i}-\bv_{t}^{i}\|^{2}\)
then
\(\gamma_{t}^i = 1\) and \(\bx_{t+1}^{i} = \bv_{t}^{i}\),
therefore we find
\begin{equation}
\label{eq:26}
\begin{aligned}
  \scal{\bg_{t}^{i}}{\bx_{t}^{i} - \bx_{t+1}^{i}}
  - \frac{M_{t+1} \|\bx_{t}^{i} - \bx_{t+1}^{i}\|^{2}}{2}
 & =
  \scal{\bg_{t}^{i}}{\bx_{t}^{i} - \bv_{t}^{i}}
  - \frac{M_{t+1} \|\bx_{t}^{i} - \bv_{t}^{i}\|^{2}}{2}
  \\
 & =
  \rho(\scal{\bg_{t}^{i}}{\bx_{t}^{i} - \bv_{t}^{i}},
  M_{t+1} \|\bx_{t}^{i} - \bv_{t}^{i}\|^{2})
  .
\end{aligned}
\end{equation}
If $\scal{\bg_t^i}{\bx_t^i-\bv_t^i}\leq
M_{t+1}\|\bx_t^i-\bv_t^i\|^2$,
then \(\gamma_{t}^i =
\scal{\bg_{t}^{i}}{\bx_{t}^{i} - \bv_{t}^{i}} /
(M_{t+1} \|\bx_{t}^{i}-\bv_{t}^{i}\|^{2})\), so
\begin{equation}
\begin{aligned}
  \scal{\bg_{t}^{i}}{\bx_{t}^{i} - \bx_{t+1}^{i}}
  - \frac{M_{t+1} \|\bx_{t}^{i} - \bx_{t+1}^{i}\|^{2}}{2}
&  =
  \frac{\scal{\bg_{t}^{i}}{\bx_{t}^{i} - \bv_{t}^{i}}^{2}}{2 M_{t+1}
    \|\bx_{t}^{i} - \bv_{t}^{i}\|^{2}}\\
&  = \rho(\scal{\bg_{t}^{i}}{\bx_{t}^{i} - \bv_{t}^{i}},
    M_{t+1} \|\bx_{t}^{i} - \bv_{t}^{i}\|^{2})
    .
\end{aligned}
\end{equation}
Summing up \eqref{e:25} for \(i \in I_{t}\) and using
subadditivity of \(\rho\) in Fact~\ref{f:rho}, we obtain
\ref{l:block-progressb}
for \(J = \varnothing\).
To show \ref{l:block-progressb} for arbitrary \(J \subseteq I \setminus I_{t}\),
we use an additional norm inequality, then
\ref{l:block-progressb} for \(J = \varnothing\),
and
Lemma~\ref{lem:max-quad} (with $c=0$), followed by subadditivity
and monotonicity of \(\rho\) \eqref{eq:1}:
\begin{align*}
  P_{t}
&  =
  \scal{\bg_{t}}{\bx_{t} - \bx_{t+1}}
  - \frac{M_{t+1} \|\bx_{t} - \bx_{t+1}\|^{2}}{2}
  +
  \frac{\|\bg_{t}^{I \setminus I_{t}}
    - \bg_{t+1}^{I \setminus I_{t}}\|^{2}}{2 M_{t+1}}
  \\
&  \geq
  \scal{\bg_{t}}{\bx_{t} - \bx_{t+1}}
  - \frac{M_{t+1} \|\bx_{t} - \bx_{t+1}\|^{2}}{2}
  +
  \frac{\scal{\bg_{t}^{J} - \bg_{t+1}^{J}}{\bx_{t}^{J} -
      \bv_{t}^{J}}^{2}}{2 M_{t+1} \|\bx_{t}^{J} - \bv_{t}^{J}\|^{2}}
  \\
&\begin{multlined}
  \geq \rho(\scal{\bg_{t}^{I_{t}}}{\bx_{t}^{I_{t}} - \bv_{t}^{I_{t}}},
  M_{t+1} \|\bx_{t}^{I_t} - \bv_{t}^{I_t}\|^{2})\\
  +
  \rho(|\scal{\bg_{t}^{J} - \bg_{t+1}^{J}}{\bx_{t}^{J} -
      \bv_{t}^{J}}|, M_{t+1} \|\bx_{t}^{J} - \bv_{t}^{J}\|^{2})
\end{multlined}
  \\
 & \geq
  \rho(|G_{I_{t} \cup J}(\bx_{t}) - \scal{\bg_{t+1}^{J}}{
      \bx_{t}^{J} - \bv_{t}^{J}}|,
    M_{t+1} \|\bx_{t}^{I_{t} \cup J} - \bv_{t}^{I_{t} \cup J}\|^{2})
  .
\end{align*}
Next, to show \ref{l:block-progressa}, we begin by using
\ref{l:block-progressb} with monotonicity of \(\rho\):
\begin{equation}
\label{e:f1}
(\forall J\subset I)\quad
P_t
  \geq
    \rho(|G_{J}(\bx_{t})
    - \scal{\bg_{t+1}^{J \setminus I_{t}}}{
      \bx_{t}^{J \setminus I_{t}}
      - \bv_{t}^{J \setminus I_{t}}}|,
    M_{t+1} D^{2})
  .
\end{equation}
Summing up \eqref{e:f1} for $k\in\{t,\ldots,t+K-1\}$
with the sets $  J_{k} \coloneqq \bigcup_{j=t+k}^{t+K-1}
I_{j}$,
we bound the righthand side
using monotonicity and subadditivity of $\rho$:
\begin{equation}
 \begin{split}
\sum_{k=0}^{K-1}P_{t+k}
&  \geq
  \sum_{k=0}^{K-1}
  \rho(|G_{J_{k}}(\bx_{t+k})
    - \scal{\bg_{t+k+1}^{J_{k} \setminus I_{t+k}}}{
      \bx_{t+k}^{J_{k} \setminus I_{t+k}}
      - \bv_{t+k}^{J_{k} \setminus I_{t+k}}}|,
    M_{t+k+1} D^{2})
  \\
&  \geq
  \rho\left(
    \tilde{G},
    \sum_{k=1}^{K} M_{t+k} D^{2}
  \right)
  ,
 \end{split}
\end{equation}
where, using Line~\ref{l:lazyg} of Algorithm~\ref{a:bcgg} and the
convention that
$G_{\varnothing}(\bx_{t+K-1})=0$,
\begin{equation*}
 \begin{split}
  \tilde{G}
&  \coloneqq
  \sum_{k=0}^{K-1}
  G_{J_{k}}(\bx_{t+k})
  - \scal{\bg_{t+k+1}^{J_{k} \setminus I_{t+k}}}{
    \bx_{t+k}^{J_{k} \setminus I_{t+k}}
    - \bv_{t+k}^{J_{k} \setminus I_{t+k}}}
  \\
&  =
  \sum_{k=0}^{K-1}
  G_{J_{k}}(\bx_{t+k}) - G_{J_{k} \setminus I_{t+k}}(\bx_{t+k+1})
  + \scal{\bg_{t+k+1}^{J_{k} \setminus I_{t+k}}}{
    \bv_{t+k}^{J_{k} \setminus I_{t+k}}
    - \bv_{t+k+1}^{J_{k} \setminus I_{t+k}}}\\
&
\begin{multlined}
=
G_{J_0}(\bx_{t}) +
  \sum_{k=1}^{K-1}
\left(  G_{J_{k}}(\bx_{t+k}) - G_{J_{k} \setminus
I_{t+k-1}}(\bx_{t+k})\right)
\\
  +   \sum_{k=0}^{K-1}\scal{\bg_{t+k+1}^{J_{k} \setminus I_{t+k}}}{
    \bv_{t+k}^{J_{k} \setminus I_{t+k}}
    - \bv_{t+k+1}^{J_{k} \setminus I_{t+k}}}
\end{multlined}
  \\
&
\begin{multlined}
  =
  G_{I_{t} \cup \dots \cup I_{t+K-1}}(\bx_{t}) % J_{0}
  +
  \sum_{k=1}^{K-1}
  G_{I_{t+k-1} \cap J_k}(\bx_{t+k})
\\
  +
  \sum_{k=0}^{K-1}
  \scal{\bg_{t+k+1}^{J_{k} \setminus I_{t+k}}}{
    \bv_{t+k}^{J_{k} \setminus I_{t+k}}
    - \bv_{t+k+1}^{J_{k} \setminus I_{t+k}}}
\end{multlined}
  \\
&  \geq
  G_{I_{t} \cup \dots \cup I_{t+K-1}}(\bx_{t})
  +
  \sum_{k=1}^{K-1}
  G_{I_{t+k-1} \cap J_k}(\bx_{t+k})
  \geq 0
  .
 \end{split}
\end{equation*}
The last two inequalities use nonnegativity of
all the summands involved,
relying on minimality of the points \(\bv_{t+k+1}\).
Finally, using monotonicity of \(\rho\) again:
\begin{equation*}
\sum_{k=0}^{K-1}P_{t+k}
  \geq
  \rho \left(
    \tilde{G}, \sum_{k=1}^{K} M_{t+k} D^{2}
  \right)
    \geq
    \rho \left(
      G_{I_{t} \cup \dots \cup I_{t+K-1}}(\bx_{t}) + A_{t},
      \sum_{k=1}^{K} M_{t+k} D^{2}
    \right)
    ,
\end{equation*}
which completes the proof.
\end{proof}

\begin{remark}[Interpretation of the gaps $A_t$ in
Lemma~\ref{lem:short-progress}]
\label{r:extra}
For every $t\in\NN$, each of the following summands in the lower
bound of Lemma~\ref{lem:short-progress}
\begin{equation}
\label{e:Aextra}
A_t=\sum_{k=1}^{K-1} G_{I_{t+k-1}\cap (I_{t+k}\cup\cdots\cup
I_{t+K-1})}(\bx_{t+k})\geq 0
\end{equation}
is a partial Frank-Wolfe gap for components
that are updated more than once between iterations $t$ and
$t+K-1$. Via Lemma~\ref{lem:pcb}, 
for each collection of reactivated components $J\subset I$,
if $f$ is convex then each summand can be bounded by
\begin{equation}
\label{e:extra}
G_{J}(\bx_{t+k})\geq f(\bx_{t+k}) -
\min_{\substack{\bx\in\bigtimes_{i\in
I}C_i\\ \bx^{I\setminus J}=\bx_{t+k}^{I\setminus J}}} f(\bx)\geq 0.
\end{equation}
As will be seen in Sections~\ref{sec:adaptive-step-size} and
\ref{sec:ss}, the gaps \(A_{t}\) contribute to faster
convergence, and they may explain the favorable behavior
observed in Section~\ref{sec:num}.
However, in general, \(A_{t}\) may not always be strictly
positive. Hence, we do not know how to utilize these gaps to
construct
a worst-case rate which is better than the cyclic-type rates of
$\mathcal{O}(K/t)$ for
convex objective functions (Section~\ref{sec:adaptive-step-size}) and
$\mathcal{O}(\sqrt{K/t})$ for nonconvex
objectives (Section~\ref{sec:ss}). We conjecture that under
additional hypotheses (potentially hemivariance, which has been
successfully used in other block-coordinate problems
\cite{Tsen01}), these gaps may lead to an improved convergence
result.
\end{remark}

\section{Convex objective functions}
\label{sec:adaptive-step-size}

In this section we show that under two step size regimes, using
Algorithm~\ref{a:bcgg} with Assumption~\ref{a:1}, the primal gap of
a convex objective function is guaranteed to
converge at a rate of $\mathcal{O}(K/t)$ after
$t$ iterations. Section~\ref{sec:flex3} is
devoted to an adaptive step size scheme whereby the constant
$L_f$ may be unknown a-priori. As a consequence, in
Section~\ref{sec:convss} we also achieve convergence for the
block-wise ``short-step'' variant of Frank-Wolfe (also sometimes
called ``adaptive'' \cite[Section~4.2]{Beck15}),
where an overestimation of $L_f$ is available.
Our convergence rates (Theorem~\ref{t:flex} and
Corollary~\ref{c:sc}) match for the special case of
cyclic activation \cite{Beck15}.

\subsection{Analysis for adaptive step sizes}
\label{sec:flex3}

In recent years, Frank-Wolfe methods have been developed to
address the situation where the smoothness constant of the
objective $L_f$ is not known. These \emph{backtracking}, or
\emph{adaptive} variants dynamically maintain an estimated smoothness
constant across iterations, typically ensuring that the smoothness
inequality \eqref{e:2} holds empirically between the current
iterate $\bx_t$ and the next iterate
$\bx_{t+1}$, at the expense of extra
gradient and/or function evaluations
\cite{Beck15,Pedr20,Poku23}. In this section, we present a
similar method for BCFW under Assumption~\ref{a:1}.

Our analysis relies on the following which, to the best of our
knowledge, first appeared in \cite{Haza16} and was later shown to
characterize convex smooth interpolability \cite{Tayl16}.
\begin{fact}[\cite{Haza16}]
\label{f:smooth3}
Let $f\colon\HHH\to\RR$ be convex and $L_f$-smooth on \(\HHH\).
Then,
\begin{equation}
  \label{eq:smooth3}
  (\forall \bx,\by \in \HHH) \quad
  f(\bx) - f(\by) - \scal{\nabla f(\by)}{\bx-\by}
  \geq
  \frac{\|\nabla f(\bx) - \nabla f(\by)\|^{2}}{2 L_{f}}.
\end{equation}
\end{fact}
The interpolability result \cite{Tayl16} implies that
a function \(f\)
satisfying \eqref{eq:smooth3} for all \(\bx, \by\) in a convex set
has an extension to a convex \(L_{f}\)-smooth function on \(\HHH\),
and therefore, for simplicity of presentation, our results in
this section assume that \(f\) is already
extended, i.e., convex and \(L_{f}\)-smooth on \(\HHH\). For 
objective functions that cannot be extended to $\HHH$, see
Remark~\ref{r:smooth}.

Fact~\ref{f:smooth3} is particularly attractive for block-iterative
algorithms, where differences of gradients often arise as error
terms, while in
\eqref{eq:smooth3} the difference appears as a lower bound on
primal progress (further demonstrated in
Lemma~\ref{lem:rec}). This
feature is the key to obtaining the same constant factors in the
convergence guarantee as for traditional Frank-Wolfe algorithms,
e.g., in \cite[Theorem~2.2]{CGSurvey}.
Hence, instead of checking the
smoothness inequality as in \cite{Pedr20} or another consequence as
in \cite{Poku23}, Algorithm~\ref{a:abcg} checks \eqref{eq:smooth3}.
Note that Algorithm~\ref{a:abcg} can be  viewed as a version of
Algorithm~\ref{a:bcgg} where $\gamma_t^i$ is computed
with an adaptive subroutine (see also Remark~\ref{r:short-motiv}).

\begin{algorithm}[t]
\caption{Adaptive Block-Coordinate Frank-Wolfe}
\label{a:abcg}  
\begin{algorithmic}[1]
\REQUIRE Function $f\colon\CCX\to\mathbb{R}$, gradient $\nabla f$,
point $\bx_0\in\CCX$, linear minimization oracles $(\lmo_i)_{i\in
I}$,
smoothness estimation \(M_{0}>0\), and approximation
parameters $0<\eta\leq 1 < \tau$.
\FOR{$t=0, 1$ \textbf{to} $\dotsc$}
\STATE Choose a nonempty block $I_t\subset I$ (See
Assumption~\ref{a:1})
\STATE $\bg_t\leftarrow \nabla f(\bx_t)$
\STATE \label{l:lazy2}
  $\widetilde{\bx}_{t+1}^i \leftarrow\bx_{t}^i$
  for all \(i\in I\setminus I_t\)
\COMMENT{Indices outside
of $I_t$ unchanged}
\STATE  $\widetilde{M}_{t+1}\leftarrow \eta M_{t}$ \COMMENT{
Candidate smoothness constant for iteration $t+1$}
\FOR{$i\in I_t$}
\label{l:block2}
\STATE $\bv_{t}^i\leftarrow\lmo_i(\bg_t^i)$\label{l:lmo2}
\STATE\label{l:flex-gam}
 $\gamma_t^i \leftarrow \min\left\{
    1,
    \dfrac{\scal{\bg_t^i}{\bx_{t}^i - \bv_{t}^i}}{\widetilde{M}_{t+1} \|\bx_{t}^i-\bv_{t}^i\|^2}
  \right\}$
\STATE\label{l:update2}
  $\widetilde{\bx}_{t+1}^i\leftarrow
\bx_{t}^i+\gamma_t^i(\bv_{t}^i-\bx_{t}^i)$
\ENDFOR
\WHILE{$ f(\bx_t) - f(\widetilde{\bx}_{t+1}) - \scal{\nabla
f(\widetilde{\bx}_{t+1})}{\bx_{t} - \widetilde{\bx}_{t+1}}
<
\|\bg_{t} - \nabla f(\widetilde{\bx}_{t+1})\|^{2}/2
\widetilde{M}_{t+1}$}
\label{l:while}
\STATE \label{l:tau} $\widetilde{M}_{t+1}\leftarrow \tau \widetilde{M}_{t+1}$
\COMMENT{If \eqref{eq:smooth3} does not hold, increase the
smoothness estimate.}
\FOR{$i\in I_t$} 
\STATE Update $\gamma_t^i$ and $\widetilde{\bx}_{t+1}^i$
as in lines~\ref{l:flex-gam} and \ref{l:update2}.
\ENDFOR
\ENDWHILE \label{l:endwhile}
\STATE \label{l:inc} $\bx_{t+1}\leftarrow \widetilde{\bx}_{t+1}$ \COMMENT{ 
Guarantees that \eqref{eq:smooth3} holds for relevant points}
\STATE $M_{t+1}\leftarrow\widetilde{M}_{t+1}$
\ENDFOR
\end{algorithmic}
\end{algorithm}

\begin{remark}
\label{r:Mbound}
By Fact~\ref{f:smooth3}, for all convex $L_f$-smooth objective
functions $f$, the loop starting at Line~\ref{l:while} of
Algorithm~\ref{a:abcg} always terminates,
at latest the first time when \(\widetilde{M}_{t+1} \geq L_{f}\), potentially overshooting by a factor of $\tau$.
Hence \(M_{t+1}\) can only be at least \(\tau L_{f}\)
if the loop terminates immediately, i.e., without any multiplication
by \(\tau\) in Line~\ref{l:tau}.
Let \(t_{0}\) be the smallest nonnegative integer
with \(\eta^{t_{0}} M_{0} \leq \tau L_{f}\),
which exists unless \(M_{0} > \tau L_{f}\) and \(\eta = 1\).
Therefore,
\begin{align}
  \label{e:Mbound}
  M_{t} &= \eta^{t} M_{0} > \tau L_{f} & 1 \leq t < t_{0} \\
  M_{t} &\leq \tau L_{f} & t \geq t_{0}
  .
\end{align}
This is consistent with the behavior of other adaptive
constants in similar works \cite{Pedr20,Poku23}: Unless $M_0$ is
initialized above $\tau L_f$, $M_t$ underestimates $\tau L_f$.
\end{remark}

\begin{remark}
\label{r:lmo-savings}
Even though the adaptive step size strategy in Algorithm~\ref{a:abcg}
requires extra
function and gradient evaluations
(Lines~\ref{l:while}--\ref{l:endwhile}), the LMOs are only computed
once per iteration, namely in
Line~\ref{l:lmo2}.
In tandem with Assumption~\ref{a:1}, this allows
for flexible management of LMO costs.
\end{remark}

The following presents a lower bound on primal progress.

\begin{lemma}[Progress bound via smoothness and convexity
\eqref{eq:smooth3}]
  \label{lem:rec}
Let
\newline
$\CCX\subset\HHH$ be a product
of $m$ nonempty compact convex sets, 
let $D$ be the diameter of $\CCX$,
let $f\colon\HHH\to\mathbb{R}$, be convex and $L_f$-smooth,
let $\rho$ be given by \eqref{e:ph},
let \(\bx^{*}\) be a solution to \eqref{e:p},
and for every nonempty $J\subset I$ let
$G_J$ be given by \eqref{e:Jgap}. 
In the setting of Algorithm~\ref{a:abcg}, 
suppose that $K$ satisfies Assumption~\ref{a:1} and
set
$A_t=\sum_{k=1}^{K-1} G_{I_{t+k-1}\cap (I_{t+k}\cup\cdots\cup
I_{t+K-1})}(\bx_{t+k})\geq 0$.
Then $(f(\bx_t))_{t\in\NN}$ is monotonically decreasing and
\begin{equation}
    \label{eq:block-progressb}
(\forall t\in\NN)\quad
    f(\bx_{t}) - f(\bx_{t+K})
    \geq
    \rho \left(
      f(\bx_{t}) - f(\bx^{*}) + A_{t},
      \sum_{k=1}^{K} M_{t+k} D^{2}
    \right)
    .
  \end{equation}
\end{lemma}
\begin{proof}
Recall from Remark~\ref{r:Mbound} that in Algorithm~\ref{a:abcg}
the loop starting at Line~\ref{l:while} terminates,
and therefore the algorithm generates an infinite sequence of iterates
satisfying the first inequality of the following chain.
The second inequality is a simple norm estimation,
and the third one is a quadratic inequality,
not needing any assumption on the scalar products and norms. We
also make use of the fact $\bx_{t}^{I\setminus
I_t}=\bx_{t+1}^{I\setminus I_t}$.
\begin{equation}
  \label{e:921}
 \begin{split}
  f(\bx_t) - f(\bx_{t+1})
  &
  \geq
  \scal{\bg_{t+1}}{\bx_{t} - \bx_{t+1}}
  +
  \frac{\|\bg_{t} - \bg_{t+1}\|^{2}}{2 M_{t+1}}
  \\
  &
  =
  \scal{\bg_{t+1}^{I_{t}}}{\bx_{t}^{I_{t}} - \bx_{t+1}^{I_{t}}}
  +
  \frac{\|\bg_{t}^{I_{t}} - \bg_{t+1}^{I_{t}}\|^{2}}{2 M_{t+1}}
  +
  \frac{\|\bg_{t}^{I \setminus I_{t}}
    - \bg_{t+1}^{I \setminus I_{t}}\|^{2}}{2 M_{t+1}}
  \\
  &
\begin{multlined}
  \geq
  \scal{\bg_{t+1}^{I_{t}}}{\bx_{t}^{I_{t}} - \bx_{t+1}^{I_{t}}}
  +
  \frac{\scal{\bg_{t}^{I_{t}} - \bg_{t+1}^{I_{t}}}{\bx_{t} -
      \bx_{t+1}}^{2}}{2 M_{t+1} \|\bx_{t}^{I_{t}} -
      \bx_{t+1}^{I_{t}}\|^{2}}\\
  +
  \frac{\|\bg_{t}^{I \setminus I_{t}}
    - \bg_{t+1}^{I \setminus I_{t}}\|^{2}}{2 M_{t+1}}
\end{multlined}
  \\
  &
  \geq
  \scal{\bg_{t}^{I_{t}}}{\bx_{t}^{I_{t}} - \bx_{t+1}^{I_{t}}}
  - \frac{ M_{t+1} \|\bx_{t}^{I_{t}} - \bx_{t+1}^{I_{t}}\|^{2}}{2}
  +
  \frac{\|\bg_{t}^{I \setminus I_{t}}
    - \bg_{t+1}^{I \setminus I_{t}}\|^{2}}{2 M_{t+1}}
  \\
  &
  =
  \scal{\bg_{t}}{\bx_{t} - \bx_{t+1}}
  - \frac{ M_{t+1} \|\bx_{t} - \bx_{t+1}\|^{2}}{2}
  +
  \frac{\|\bg_{t}^{I \setminus I_{t}}
    - \bg_{t+1}^{I \setminus I_{t}}\|^{2}}{2 M_{t+1}}\\
 & = P_{t}
  ,
 \end{split}
\end{equation}
where \(P_{t}\) is the same as in Lemma~\ref{lem:short-progress}.
Monotonicity of $(f(\bx_t))_{t\in\NN}$ follows from
Lemma~\ref{lem:short-progress}\ref{l:block-progressb}.
Telescoping the lefthand sum of
\eqref{e:921} and invoking\newline
Lemma~\ref{lem:short-progress}\ref{l:block-progressa}
with Assumption~\ref{a:1}, we find
$f(\bx_t)-f(\bx_{t+K})\geq
\rho(G_{I}(\bx_t)+A_t,\sum_{k=1}^KM_{t+k}D^2)$.
Since $f$ is convex, by optimality of the LMO and \eqref{e:1}, we
have
\begin{equation}
G_{I}(\bx_{t})
\geq\scal{\nabla
f(\bx_t)}{\bx_t-\bx^*}\geq
f(\bx_{t}) - f(\bx^{*}),
\end{equation}
so
\eqref{eq:block-progressb} follows from
monotonicity of \(\rho\) (Fact~\ref{f:rho}).
\end{proof}

\begin{theorem}
\label{t:flex}
Let $\CCX\subset\HHH$ be a product
of $m$ nonempty compact convex sets,
let $D$ be the diameter of $\CCX$,
let $f\colon\HHH\to\mathbb{R}$ be convex and $L_f$-smooth,
let $\tau>1\geq\eta>0$ and $M_0>0$ be approximation parameters, 
let $\bx^*$ be a solution to \eqref{e:p},
and for every nonempty $J\subset I$ let
$G_J$ be given by \eqref{e:Jgap}. 
If \(\eta = 1\), we assume \(M_{0} \leq \tau L_{f}\)
and set \(n_{0} = 0\)\footnote{If $M_0>\tau
L_f$, then $f$ is also $M_0$-smooth, so this assumption is WLOG 
for notational convenience in the case $\eta=1$.};
otherwise, 
$n_0 \coloneqq \max\{
\lceil \log (\tau L_{f} / (\eta M_{0})) / (K \log \eta) \rceil, 0\}$.
In the setting of Algorithm~\ref{a:abcg}, 
suppose that $K$ satisfies Assumption~\ref{a:1}, and 
set
\[A_t=\sum_{k=1}^{K-1} G_{I_{t+k-1}\cap (I_{t+k}\cup\cdots\cup
I_{t+K-1})}(\bx_{t+k})\geq 0.\]
Then, in the first \(t\) iterations, Algorithm~\ref{a:abcg}
evaluates $f$ and $\nabla f$ at most
$t+1+\max\{0, \lceil
\log_{\tau}(\eta^{-t}L_f/M_0)\rceil\}$
times. Furthermore, for every $n\in\NN$,
$  f(\bx_{nK}) - f(\bx^*)$ is bounded above by
\begin{equation}
\label{e:cv-rate}
  \begin{cases}
\underset{0\leq p \leq n-1}{\min}\left\{
    \dfrac{K \eta^{p K} M_{0} D^{2}}{2}
-A_{pK}\right\}
    &\text{if } 1 \leq n \leq n_{0}+1 \\[1.5em]
    \dfrac{2 K \tau L_f D^2}{n - n_{0}+\sum_{p=n_0}^n
      \frac{2 A_{pK}}{f(\bx_{n_0})-f(\bx^*)}
      + \left(\frac{A_{pK}}{f(\bx_{n_0})-f(\bx^*)}\right)^2}
    &\text{if } n > n_{0} + 1
    .
  \end{cases}
\end{equation}
\end{theorem}
\begin{proof}
We start by estimating the number of function and gradient
computations of Algorithm~\ref{a:abcg}.
Except for \(t=0\), where \(f(\bx_{0})\) and \(\nabla f(\bx_{0})\)
are computed, for all $t\geq 1$,
in the preceding iteration
$f(\bx_t)$ and $\nabla f(\bx_t)$
have already been computed.
So, in the first \(t\) iterations,
there have been $t+1$ function and gradient evaluations
for the \emph{initial} check of Line~\ref{l:while} in each iteration.
Now, let $k$ denote the total number of function
and gradient evaluations in the first \(t\) iterations, i.e., $k-t-1$
is the total number \emph{subsequent} checks of Line~\ref{l:while}
and
also the number of times that line~\ref{l:tau} has been
executed.
By Remark~\ref{r:Mbound}, unless \(k=0\), we have
\(M_{t} = \eta^t \tau^{k-t-1} M_0 < \tau L_{f}\),
therefore at most \(k \leq t+1+\max\{0, \lceil
\log_{\tau}(\eta^{-t}L_f/M_0)\rceil\}\)
function and gradient evaluations are performed.

We turn now to the convergence rate.
As in Remark~\ref{r:Mbound},
let \(t_{0}\) be the smallest nonnegative integer
with \(\eta^{t_{0}} M_{0} \leq \tau L_{f}\).
The number \(n_{0}\) is chosen to be the smallest nonnegative integer
with \(t_{0} \leq n_{0} K + 1\).
Let \(1 \leq n \leq n_{0}\).
By Remark~\ref{r:Mbound}, \(M_{(n - 1) K + 1}=\eta^{(n - 1) K +
1}M_0 > \tau L_{f}\)
and
\(M_{(n - 1) K + 1} \geq M_{t}\) for all \(t > (n - 1) K\).
By Lemma~\ref{lem:rec} and Fact~\ref{f:rho},
\begin{equation}
\label{e:421}
 \begin{split}
  f(\bx_{(n-1) K}) - f(\bx_{n K})
  &
  \geq
  \rho\left(
    f(\bx_{(n - 1) K}) - f(\bx^{*})+A_{(n-1)K},
    \sum_{k=1}^{K} M_{n K + k} D^{2}
  \right)
  \\
  &
  \geq
  \rho(f(\bx_{(n - 1) K}) - f(\bx^{*})+A_{(n-1)K},
  K \eta^{(n - 1) K + 1} M_{0} D^{2})
  \\
  &
  \geq
  f(\bx_{(n - 1) K}) - f(\bx^{*})+A_{(n-1)K}
  -
  \frac{K \eta^{(n - 1) K+1} M_{0} D^{2}}{2}
  .
 \end{split}
\end{equation}
Rearranging \eqref{e:421} shows \(f(\bx_{n K}) - f(\bx^{*})
\leq K \eta^{(n - 1) K + 1} M_{0} D^{2} / 2 -A_{(n-1)K}\).
Therefore, since $(f(\bx_t))_{t\in\NN}$ is monotonically
decreasing (Lemma~\ref{lem:rec}), the first case of \eqref{e:cv-rate}
follows.
By the choice of \(n_{0}\),
we have \(\eta^{n_{0} K + 1} M_{0} \leq \tau L_{f}\),
thus \(M_{t} \leq \tau L_{f}\) for \(t \geq n_{0} K \).
Let \(n \geq n_{0} + 1\).
Then Lemma~\ref{lem:rec} yields
\begin{equation}
 \begin{aligned}
  f(\bx_{(n-1) K}) - f(\bx_{n K})
  &
  \geq
  \rho\left(
    f(\bx_{(n - 1) K}) - f(\bx^{*}),
    \sum_{k=1}^{K} M_{(n- 1) K + k} D^{2}
  \right)
  \\
  &
  \geq
  \rho(f(\bx_{(n - 1) K}) - f(\bx^{*}), K \tau L_{f} D^{2})
  .
 \end{aligned}
\end{equation}
and the second case of \eqref{e:421} follows from 
Lemma~\ref{l:recursion}.
\end{proof}

To interpret the extra gaps $(A_t)_{t\in\NN}$ in
Theorem~\ref{t:flex}, see
Remark~\ref{r:extra}.
\begin{corollary}
  In the context of Theorem~\ref{t:flex},
  let Algorithm~\ref{a:abcg}
  use a block selection strategy
  without coordinate reactivation, i.e.,
  \(I_{n K + i} \cap I_{n K + j} = \varnothing\)
  for all \(n\) and \(1 \leq i < j \leq K\).
  Then,
for any \(0 < \varepsilon \leq K \tau L_{f} D^{2} / 2\),
the primal gap \(f(\bx_{nK}) - f(\bx^{*}) \leq \varepsilon\)
is guaranteed
after at most $m(n_0+ \frac{2K\tau L_f D^2}{\varepsilon})$ 
LMO calls
and computation of
at most
$t+1+\max\{0, \lceil
\log_{\tau}(\eta^{-t}L_f/M_0)\rceil\}$ function values and gradients.
\end{corollary}

\begin{remark}
\label{r:linear}
Under stricter assumptions, one can achieve linear convergence by
following the template \cite[Section~2.2.1]{CGSurvey} from 
the penultimate inequality in the proof of Lemma~\ref{lem:rec}.
\end{remark}

\subsection{Short-steps with convex objectives}
\label{sec:convss}

In this section, we consider the step size rule \eqref{e:shortA}
of Remark~\ref{r:short-motiv}.

\begin{algorithm}[H]
\caption{Block-Coordinate Frank-Wolfe (BCFW) with Short-Steps}
\label{a:bcg}  
\begin{algorithmic}[1]
\REQUIRE Function $f\colon\CCX\to\mathbb{R}$, gradient $\nabla f$,
point $\bx_0\in\CCX$, linear minimization oracles $(\lmo_i)_{i\in
I}$
\FOR{$t=0, 1$ \textbf{to} $\dotsc$}
\STATE Choose a nonempty block $I_t\subset I$
\STATE $\bg_t\leftarrow \nabla f(\bx_t)$
\FOR{$i=1$ \textbf{to} $m$}
\IF{$i\in I_t$}
\label{l:block}
\STATE $\bv_{t}^i\leftarrow\lmo_i(\bg_t^i)$
\STATE\label{l:step}
$\gamma_t^i
 \leftarrow \min\left\{ 1, 
    \dfrac{\scal{\bg_t^i}{\bx_{t}^i - \bv_{t}^i}}{L_f
     \|\bv_{t}^i-\bx_{t}^i\|^2}
\right\}$ 
\STATE \label{l:update} $\bx_{t+1}^i \leftarrow
\bx_{t}^i+\gamma_t^i(\bv_t^i-\bx_t^i)$
\ELSE
\STATE \label{l:lazy}
$\bx_{t+1}^i\leftarrow \bx_{t}^i$
\ENDIF
\ENDFOR
\ENDFOR
\end{algorithmic}
\end{algorithm}

Short-Step BCFW (Algorithm~\ref{a:bcg}) requires an upper bound on
$L_f$. For this price, the algorithm becomes easier to
parallelize in lines~\ref{l:block}--\ref{l:update}, foregoes any
function evaluations, and requires only one
gradient evaluation per iteration. Also, both the convergence rate
and the prefactor of obtained in this section match the non-block
version ($K=1$) \cite[Theorem~2.2]{CGSurvey}.

\begin{corollary}
\label{c:sc}
Let $\CCX\subset\HHH$ be a product
of $m$ nonempty compact convex sets,
let $D$ be the diameter of $\CCX$,
let $f\colon\HHH\to\mathbb{R}$ be convex and $L_f$-smooth,
let $\bx^*$ be a solution to \eqref{e:p},
and for every nonempty $J\subset I$ let
$G_J$ be given by \eqref{e:Jgap}.
In the setting of Algorithm~\ref{a:bcg}, 
suppose that $K$ satisfies Assumption~\ref{a:1}, and
set
$A_t=\sum_{k=1}^{K-1} G_{I_{t+k-1}\cap (I_{t+k}\cup\cdots\cup
I_{t+K-1})}(\bx_{t+k})\geq 0$. Then, for every $n\in\NN$,
\begin{equation}
\label{e:sc}
  f(\bx_{nK}) - f(\bx^*)
  \leq
  \begin{cases}
    \dfrac{K L_f D^{2}}{2}
-A_{0}
    &\text{if } n=1 \\
    \dfrac{2 K L_f D^2}{n - 1+\sum_{p=1}^n
      \frac{2 A_{pK}}{f(\bx_{1}) - f(\bx^{*})}
      + \left(\frac{A_{pK}}{f(\bx_{1}) - f(\bx^{*})}\right)^2}
    &\text{if } n \geq 2
    .
  \end{cases}
\end{equation}
Furthermore Algorithm~\ref{a:bcg} requires one gradient evaluation
per iteration.
\end{corollary}
\begin{proof}
This follows from the fact that Algorithm~\ref{a:bcg} produces
the same sequence of iterates as Algorithm~\ref{a:abcg}: by
initializing Algorithm~\ref{a:abcg} with $M_0=L_f$ and 
$\eta=1$, as by Fact~\ref{f:smooth3}, the
condition in Line~\ref{l:while} of Algorithm~\ref{a:abcg}
is always true.
Hence,
this case of Algorithm~\ref{a:abcg} coincides with Algorithm \ref{a:bcg} and
we achieve convergence from Theorem~\ref{t:flex} for all $\tau>1$;
taking the limit as $\tau\searrow 1$ yields \eqref{e:sc}.
Clearly, Algorithm~\ref{a:bcg} requires one gradient evaluation per
iteration.
\end{proof}

\section{Nonconvex objective functions}
\label{sec:ss}

In this section, we consider
Algorithm~\ref{a:bcg} under Assumption~\ref{a:1} on
nonconvex objective functions with $L_f$-Lipschitz continuous
gradients. Since \eqref{eq:smooth3} only holds for
smooth and convex functions, a different progress lemma which
relies on the traditional smoothness inequality \eqref{e:2} is
derived. 
We begin with a blockwise descent lemma.

\begin{lemma}
\label{l:d}
Let $\CCX\subset\HHH$ be a product of $m$ nonempty compact
convex sets and let $f\colon\HHH\to\mathbb{R}$ be $L_f$-smooth on
$\CCX$.
In the setting of Algorithm~\ref{a:bcg}, $(f(\bx_t))_{t\in\NN}$ is
monotonically decreasing, and
  \begin{equation}
    \label{eq:short-gap}
(\forall t\in\NN)\quad
    f(\bx_{t}) - f(\bx_{t+1})
    \geq
    \frac{\scal{ \nabla f(\bx_{t})}{\bx_{t} - \bx_{t+1}}}{2}
    \geq \frac{L_{f} \|\bx_{t} - \bx_{t+1}\|^{2}}{2}.
  \end{equation}
\end{lemma}
\begin{proof}
By \eqref{e:shortA},
for every \(i \in I_{t}\), $
L_f\|\bx_t^i-\bv_t^i\|\gamma_t^i\leq G_i(\bx_t)
$, so
\begin{equation}
\label{e:44}
  \scal{\nabla f(\bx_{t})}{\bx_{t}^{i} - \bx_{t+1}^{i}}
  =
  \gamma_{t}^{i} G_{i}(\bx_{t})
  \geq
  (\gamma_{t}^{i})^{2} L_{f}  \|\bv_{t}^{i}-\bx_{t}^{i}\|^{2}
  =
  L_{f}  \|\bx_{t}^{i} - \bx_{t+1}^{i}\|^{2}
  .
\end{equation}
Summing \eqref{e:44} for all \(i \in I_{t}\)
and using $\bx_{t}^{i}=\bx_{t+1}^{i}$ for $i \notin I_{t}$
(Line~\ref{l:lazy}),
we obtain
\begin{equation}
  \label{eq:2}
  \scal{ \nabla f(\bx_{t})}{\bx_{t} - \bx_{t+1}}
  \geq
  L_{f}  \|\bx_{t} - \bx_{t+1}\|^{2}
  .
\end{equation}
We combine this with the smoothness inequality \eqref{e:2} to
derive the claim:
\begin{equation}
 \begin{split}
  f(\bx_{t}) - f(\bx_{t+1})
  &
  \geq
  \scal{\nabla f(\bx_{t})}{\bx_{t}-\bx_{t+1}}
  - \frac{L_{f}}{2}\|\bx_{t} - \bx_{t+1}\|^{2}
  \\
  &
  \geq
  \frac{\scal{\nabla f(\bx_{t})}{\bx_{t}-\bx_{t+1}}}{2}
  \\
  &
  \geq
  \frac{L_{f}}{2}\|\bx_{t} - \bx_{t+1}\|^{2}
  .
 \end{split}
\end{equation}
\end{proof}

\begin{lemma}[Progress bound via smoothness \eqref{e:2}]
  \label{lem:batch-progress}
Let $\CCX\subset\HHH$ be a product of $m$ nonempty compact
convex sets,
let $D$ be the diameter of $\CCX$,
let $f\colon\CCX\to\mathbb{R}$ be a function with
$L_f$-Lipschitz continuous gradient \(\nabla f\) on $\CCX$,
let $G_I$ be given by \eqref{e:fwgap}, 
and let $t\in\NN$.
In the setting of Algorithm~\ref{a:bcg}, suppose that
$K$ satisfies
Assumption~\ref{a:1}, and set
\newline
$A_{t}=\sum_{k=0}^{K-1}G_{(I_{t+k}\cup\cdots\cup I_{t+K-1})\cap
I_{t+k-1}}(\bx_{t+k})\geq 0$. Then
  \begin{equation}
    \label{eq:batch-progress}
    f(\bx_{t}) - f(\bx_{t+K})
    \geq
    \frac{\rho(G_{I}(\bx_{t}) + A_{t}, K L_{f} D^{2})}{2}
    .
  \end{equation}
\end{lemma}
\begin{proof}
For any iteration \(t\), we have by smoothness \eqref{e:2}
\begin{equation}
  \label{eq:5}
  f(\bx_{t}) - f(\bx_{t+1})
  \geq
  \scal{\bg_{t}}{\bx_{t} - \bx_{t+1}}
  - \frac{L_{f} \|\bx_{t} - \bx_{t+1}\|^{2}}{2}
  .
\end{equation}
By Lemma~\ref{l:d} and Lipschitz continuity of gradient, 
we also have
\begin{equation}
  \label{eq:6}
  f(\bx_{t}) - f(\bx_{t+1})
       \geq \frac{L_{f} \|\bx_{t} - \bx_{t+1}\|^{2}}{2}
       \geq \frac{\|\bg_{t} - \bg_{t+1}\|^{2}}{2L_f}
       \geq \frac{\|\bg_{t}^{I\setminus I_t} -
\bg_{t+1}^{I\setminus I_t}\|^{2}}{2L_f}
  .
\end{equation}
The sum of \eqref{eq:5} and \eqref{eq:6} is
\begin{equation}
\label{e:821}
  2 (f(\bx_{t}) - f(\bx_{t+1}))
  \geq
  \scal{\bg_{t}}{\bx_{t} - \bx_{t+1}}
  - \frac{L_{f} \|\bx_{t} - \bx_{t+1}\|^{2}}{2}
+
\frac{\|\bg_{t}^{I\setminus I_t} -
\bg_{t+1}^{I\setminus I_t}\|^{2}}{2L_f}
  .
\end{equation}
Summing
\eqref{e:821} from $t$ to $t+K-1$, 
invoking
Lemma~\ref{lem:short-progress}\ref{l:block-progressa}, then
dividing by $2$ yields \eqref{eq:batch-progress}.
\end{proof}

We are ready to provide convergence for nonconvex functions.
Due to lack of optimality guarantees for nonconvex functions,
a typical result for Frank-Wolfe algorithms states that the algorithm
produces a point with arbitrarily small F-W gap
\cite{Pedr20,CGSurvey},
this is closely related to stationarity.

\begin{theorem}[Nonconvex convergence]
\label{t:ncv-conv}
Let $\CCX\subset\HHH$ be a product of $m$ nonempty compact
convex sets with diameter $D$. Let $f\colon\HHH\to\mathbb{R}$ 
be such that $\nabla f$ is $L_f$-Lipschitz continuous on
$\CCX$. Let $G_I$ be given by
\eqref{e:fwgap}. In the setting of
Algorithm~\ref{a:bcg}, suppose that $K$ satisfies
Assumption~\ref{a:1},
set $H_0=f(\bx_0)-\inf_{\bx\in\CCX}f(\bx)$, and for every
$n\in\NN$ set
\begin{equation}
A_n=\sum_{k=1}^{K-1} G_{I_{n+k-1}\cap (I_{n+k}\cup\cdots\cup
I_{n+K-1})}(\bx_{n+k})\geq 0.
\end{equation}
Then, for every $n\in\NN\setminus\{0\}$, 
\begin{equation}
\label{e:ncv-conv2}
\begin{aligned}
\min_{0\leq p\leq n-1}G_{I}(\bx_{pK})
&
\leq
\frac{1}{n}\sum_{p=0}^{n-1}G_{I}(\bx_{pK})
\\
&
\leq
\begin{cases}
  \frac{2 H_0 - \sum_{p=0}^{n-1}A_{pK}}{n} + \frac{KL_fD^2}{2}
  &\text{if } n \leq \frac{2 H_0}{KL_fD^2}\\
  2 D\sqrt{\frac{H_0 K L_f}{n}}
  - \frac{\sum_{p=0}^{n-1}A_{pK}}{n}
  &\text{otherwise.}
\end{cases}
\end{aligned}
\end{equation}
In consequence, there exists a subsequence $(n_k)_{k\in\NN}$ such
that $G_I(\bx_{n_kK})\to 0$, and every accumulation point
of $(\bx_{n_kK})_{k\in\NN}$ is a stationary point of 
\eqref{e:p}.
\end{theorem}
\begin{proof}
By telescoping the result of Lemma~\ref{lem:batch-progress} over
multiples of $K$, then using subadditivity \eqref{eq:1},
\begin{equation*}
\begin{split}
2 H_0\geq 2(f(x_0)-f(x_{nK})) &\geq
\sum_{p=0}^{n-1}\rho(G_I(\bx_{pK})+A_{pK}, KL_fD^2)
\\
&
\geq \rho\left(\sum_{p=0}^{n-1}G_I(\bx_{pK})+A_{pK}, nKL_fD^2\right).
\end{split}
\end{equation*}
Observe that, for $x,y\geq 0$ and $b>0$, we have
for $y\geq \frac{b}{2}$
by strict monotonicity of \(\rho\) that
$y\geq\rho(x,b)$ if and only if 
$x\leq y+\frac{b}{2}$.
For $y\leq \frac{b}{2}$
we have $y\geq\rho(x,b)$ if and only if
$x \leq \sqrt{2by}$.
Therefore,
\begin{equation}
\label{e:49}
\sum_{p=0}^{n-1}G_I(\bx_{pK})+A_{pK}\leq
\begin{cases}
2 H_0 + \frac{nKL_fD^2}{2} &\text{if } 2 H_0\geq nKL_fD^2\\
2 D\sqrt{H_0 n K L_f} &\text{otherwise.}
\end{cases}
\end{equation}
Dividing \eqref{e:49} by $n$ and rearranging
yields \eqref{e:ncv-conv2}.
\end{proof}

\begin{remark}
\label{r:smooth}
It is possible to still achieve $\mathcal{O}(K/t)$ convergence, by
replacing Line~\ref{l:while} in Algorithm~\ref{a:abcg} with an
analogue of \eqref{e:2} (as in \cite{Pedr20}); the proof proceeds
by a similar argument to that of Theorem~\ref{t:flex},
replacing Lemma~\ref{lem:rec} with Lemmas \ref{l:d} and
\ref{lem:batch-progress}, and
yielding the same rate with a worse constant.
This may be useful in situations where checking \eqref{e:2} is
preferable to \eqref{eq:smooth3} (e.g., Section~\ref{sec:ex3}), or
if $f$ is convex and $\nabla f$ is Lipschitz-continuous on $\CCX$,
yet $f$ is not extendable to a smooth function on $\HHH$ (see
Fact~\ref{f:smooth3}).
\end{remark}

\section{Numerical Experiments}
\label{sec:num}

In this section we examine different block selection strategies
covered by Assumption~\ref{a:1}; Sections~\ref{sec:ex1} and
\ref{sec:ex2} contain simple experiments where
the LMO for the last constraint is far more expensive than the
other LMOs, while Section~\ref{sec:ex3} uses BCFW to solve a structural SVM training
problem \cite{Laco13}. In line with Theorems~\ref{t:flex} and
\ref{t:ncv-conv}, our optimality criterion is the primal gap for
the convex problems in Sections~\ref{sec:ex1} and \ref{sec:ex3},
and minimal F-W gap for the nonconvex problem in
Section~\ref{sec:ex2}. 
Computations were performed on a single Slurm node with
3~GB RAM and no concurrencies,
allocated on an Intel Xeon Gold~6338 machine with 2.0~GHz
CPU speed, running Linux. For each experiment, we averaged the results over 20 trials.

In Sections~\ref{sec:ex1} and \ref{sec:ex2}, we ran $10,000$
iterations of BCFW using the package \texttt{FrankWolfe.jl}
(v0.3.3) \cite{Besa22} in Julia~1.11.5.
To obtain feasible initial iterates $\bx_0$,
we evaluate the LMO for each component in directions with $\mathcal{N}(0,1)$-distributed entries.
Within both experiments,
we vary how the blocks $(I_t)_{t\in\NN}$ are
selected. We compare block selection strategies newly allowed by
Assumption~\ref{a:1} to the following techniques
(e.g., in \cite{Patr98,Beck15}) covered by our results:
\begin{enumerate}
\item 
\label{act:1}
\emph{Full} activation: $I_t = I$.
\item
\label{act:2}
\emph{Cyclic} activation: $I_t= \{t\}\pmod m$.
\item
\label{act:3}
\emph{P-Cyclic} activation: $I_t=\sigma_{\lfloor t/m\rfloor}( t\pmod
m),$ where for every cycle over $m$ iterations, $\sigma_{\lfloor
t/m\rfloor}$ is a uniformly random permutation of $\{1,\ldots,m\}$.
\item
\label{act:4}
Lazy essentially cyclic activation, 
\emph{E-Cyclic}: $I_t=\{\mathfrak{i}(t)\}$ satisfying
\eqref{e:ecyclic}, (which requires $K\geq m$); each sequence of $K$
components is created by randomly shuffling tuples of
$\{1,\ldots,m-1\}$, with the final activation being the expensive
$m$th component.
\end{enumerate} 
Among these selection methods, only E-cyclic allows for the user to
leverage a-priori knowledge of LMO runtimes.

In addition to plotting optimality criterion against iterations
and time, we also include plots against the number of calls to the
most computationally-intensive LMO (spectrahedron LMO for
Section~\ref{sec:ex1}; nuclear norm ball LMO for
Section~\ref{sec:ex2}). The expensive-LMO count is a more
reproducible proxy for time in both experiments, since
it correlated with time used for all algorithms, and it was the
dominant cost of even a full F-W iteration. 
These plots are unfair to the
full/cyclic selection schemes, since they are forced to activate
the most expensive LMO at a fixed rate and our new methods have
more flexibility to re-activate cheaper components; however,
flexibility in activation is precisely the point here, and until
this work it was unclear if such reactivations would
provide progress for BCFW at all.

In Section~\ref{sec:ex3}, we compare the original algorithm of ~\cite{Laco13} to Algorithm~\ref{a:abcg} in MATLAB R2023b using the same initial iterates.
For block selection strategies, we consider a generalized P-Cyclic activation with blocks of size $n$:
\[I_t = \{\sigma_{\lfloor j/m\rfloor}( j\bmod m) \mid j \in \{(t-1)n +1, \dots, tn\} \},\]
where each $\sigma$ is a random permutation of $\{1,\ldots,m\}$.
These activations satisfy Assumption~\ref{a:1} for $K = 2 \lceil n/m \rceil$.
The algorithm in~\cite{Laco13} uses a P-Cyclic activation with
singleton blocks as well as a uniform activation, i.e.,
$I_t = \{\sigma_t (t\pmod m)\}$, which does not satisfy
Assumption~\ref{a:1}.
To ensure a fair comparison of the orginal algorithm and Algorithm~\ref{a:abcg}, we run both algorithms for $150$ epochs,
i.e., $150$ LMO calls for each component.
This yields different iteration counts depending on the chosen block sizes.
We plot against number of iterations and epochs as well as wall-clock time.

The code used to produce the result in Section~\ref{sec:ex1} and
\ref{sec:ex2} can be found at
\url{https://github.com/zevwoodstock/BlockFW}.
The code used to produce the results in Section~\ref{sec:ex3} can be found at
\url{https://github.com/JannisHal/BCFWstruct}.
The new block updating schemes and Algorithm~\ref{a:abcg} are available in
\url{https://github.com/ZIB-IOL/FrankWolfe.jl}
as \texttt{LazyUpdate} and \texttt{adaptive\_block\_coordinate\_frank\_wolfe}.

\subsection{Experiment 1: Intersection problem}
\label{sec:ex1}
The goal is to find a matrix $x\in\mathbb{R}^{s\times s}$ in
the intersection of the hypercube $C_1=[-1,1/s]^{s\times s}$ and
the spectraplex $C_2=\{x\in\mathbb{R}^{s\times s} \mid x\succeq
0, \operatorname{Trace}(x)=1\}$
for various values of $s$.
The convex sets are selected to have a thin intersection,
and hence the minimal value of
\begin{equation}
\label{e:n1}
\minimize_{\bx\in C_1\times
C_2}\frac{1}{2}\|\bx^1-\bx^2\|^2
\end{equation}
is zero. Problem \eqref{e:n1} is convex, with smoothness constant
$L_f = 2$. In this
problem, the spectrahedral linear minimization oracle, $\lmo_2$, is
far more expensive than
$\lmo_1$. So, for this experiment we compare the traditional BCFW
activations \ref{act:1}--\ref{act:3} with the following
``$q$-lazy'' scheme which is newly allowed for BCFW by
Assumption~\ref{a:1} (with $K=q$) and has
improved computational performance in proximal algorithms
\cite{Comb22}:
\begin{equation}
\label{e:lazy}
(\forall t\in\NN)\quad I_t=\begin{cases}
\{1,2\} &\text{if}\;\;t\equiv 0 \;\;\text{mod} \;\; q;\\
\{1\} &\text{otherwise}.
\end{cases}
\end{equation}
We run $20$
instances of this problem on random
initializations, and the averaged results are shown in
Figure~\ref{fig:ex1-shortstep} for short-step (Algorithm~\ref{a:bcg}).
Even though using
$q$-lazy activation is computationally cheaper on average, the
per-iteration progress is still competitive with that of full,
cyclic, P-cyclic, and E-cyclic activation. However, since
they compute
$\lmo_2$ at a much lower rate, these activation strategies also
have a faster per-iteration computation time. The activation scheme
\eqref{e:lazy} performs similarly to an E-cyclic rule with a
similar refresh rate; this may be due to the small number of
components in this problem.

In Figure~\ref{fig:ex1-comparison} we compare BCFW with short-step to the adaptive variant
(Algorithm~\ref{a:abcg}).
The results indicate that, at least on problems
like~\eqref{e:n1}, i.e.,
with a small number of components and where $L_f$ is known,
short-step may be preferable.
However, some updating strategies (e.g., full and cyclic) may exhibit performance
improvements, while others exhibit relatively little improvement
change. See Section~\ref{sec:ex3} for a better use-case for
Algorithm~\ref{a:abcg}.

\begin{figure}[htb]
\pgfplotslegendfromname{legend_conv_ex1}
\begin{center}
\subfloat[][$s=100$]{
\centering
\begin{tikzpicture}[scale=1.0]
\begin{axis}[height=4.2cm,width=4.1cm, legend cell align={left},
ylabel style={yshift=-0.75em,font=\scriptsize}, 
xlabel style={yshift=0.50em,font=\scriptsize},
tick label style={font=\tiny},
minor tick num=0,
legend columns=2,
%legend entries={Full,$5$-Lazy,Cyclic,$10$-Lazy,P-Cyclic,$20$-Lazy},
%legend style={
%    font=\tiny,
%    at={(-0,1.4)},
%    anchor=north west,
%    },
xlabel=Time (sec), 
ylabel=$f(x)-f(x^*)$,
grid,
xmin =0, xmax=35, ymin=5e-3, ymax=1e1, ymode=log, mark repeat=18]
\addplot+[thick, mark=triangle, mark repeat=15, color=cb-burgundy] table[x={time}, y={primal}] {BCFW_A20_100_full-\dateconvex.txt};
\addplot+[thick, mark=oplus, color=cb-brown]
table[x={time}, y={primal}] {BCFW_A20_100_custom5-\dateconvex.txt};
\addplot+[thick, mark=square, color=cb-blue] table[x={time}, y={primal}] {BCFW_A20_100_cyclic-\dateconvex.txt};
\addplot+[thick, mark=star, color=cb-lilac]
table[x={time}, y={primal}] {BCFW_A20_100_custom10-\dateconvex.txt};
\addplot+[thick, mark=o, color=cb-green-sea] table[x={time}, y={primal}] {BCFW_A20_100_stoc-\dateconvex.txt};
\addplot+[thick,solid, mark=otimes, mark repeat=30, color=cb-blue-light]
table[x={time}, y={primal}] {BCFW_A20_100_ecyc20-\dateconvex.txt};
\addplot+[thick,solid, mark=diamond, color=cb-salmon-pink]
table[x={time}, y={primal}] {BCFW_A20_100_custom20-\dateconvex.txt};
\end{axis}
\end{tikzpicture}
}
\subfloat[][$s=300$]{
\centering
\begin{tikzpicture}[scale=1.0]
\begin{axis}[height=4.2cm,width=4.1cm, legend cell align={left},
ylabel style={yshift=-0.75em,font=\scriptsize}, 
xlabel style={yshift=0.50em,font=\scriptsize},
tick label style={font=\tiny},
xlabel=Time (sec),
grid,
xmin =0, xmax=320, ymin=1e-2, ymax=1e2, ymode=log, mark repeat=18]
\addplot+[thick, mark=triangle, mark repeat=15, color=cb-burgundy] table[x={time}, y={primal}] {BCFW_A20_300_full-\dateconvex.txt};
\addplot+[thick, mark=oplus, color=cb-brown]
table[x={time}, y={primal}] {BCFW_A20_300_custom5-\dateconvex.txt};
\addplot+[thick, mark=square, color=cb-blue] table[x={time}, y={primal}] {BCFW_A20_300_cyclic-\dateconvex.txt};
\addplot+[thick, mark=star, color=cb-lilac]
table[x={time}, y={primal}] {BCFW_A20_300_custom10-\dateconvex.txt};
\addplot+[thick, mark=o, color=cb-green-sea] table[x={time}, y={primal}] {BCFW_A20_300_stoc-\dateconvex.txt};
\addplot+[thick,solid, mark=otimes, mark repeat=30, color=cb-blue-light]
table[x={time}, y={primal}] {BCFW_A20_300_ecyc20-\dateconvex.txt};
\addplot+[thick,solid, mark=diamond, color=cb-salmon-pink]
table[x={time}, y={primal}] {BCFW_A20_300_custom20-\dateconvex.txt};
\end{axis} \end{tikzpicture}
}
\subfloat[][$s=500$]{
\centering
\begin{tikzpicture}[scale=1.0]
\begin{axis}[height=4.2cm,width=4.1cm, legend cell align={left},
ylabel style={yshift=-0.75em,font=\scriptsize}, 
xlabel style={yshift=0.50em,font=\scriptsize},
tick label style={font=\tiny},
legend columns=4,
legend
entries={Full,Cyclic,P-Cyclic,E-Cyclic,$5$-Lazy,$10$-Lazy,$20$-Lazy},
legend style={
    font=\scriptsize,
    at={(1.0,1.75)},
    anchor=north east,
    },
legend to name={legend_conv_ex1},
ytick={1e-1,1,1e1},
xlabel=Time (sec),
%ylabel=$f(x)-f(x^*)$,
grid,
xmin =0, xmax=1300, ymin=4e-2, ymax=1e2, ymode=log, mark repeat=18]
\addplot+[thick, mark=triangle, mark repeat=15, color=cb-burgundy] table[x={time}, y={primal}] {BCFW_A20_500_full-\dateconvex.txt};
\addplot+[thick, mark=square, color=cb-blue] table[x={time}, y={primal}] {BCFW_A20_500_cyclic-\dateconvex.txt};
\addplot+[thick, mark=o, color=cb-green-sea] table[x={time}, y={primal}] {BCFW_A20_500_stoc-\dateconvex.txt};
\addplot+[thick,solid, mark=otimes, mark repeat=30, color=cb-blue-light]
table[x={time}, y={primal}] {BCFW_A20_500_ecyc20-\dateconvex.txt};
\addplot+[thick, mark=oplus, color=cb-brown]
table[x={time}, y={primal}] {BCFW_A20_500_custom5-\dateconvex.txt};
\addplot+[thick, mark=star, color=cb-lilac]
table[x={time}, y={primal}] {BCFW_A20_500_custom10-\dateconvex.txt};
\addplot+[thick,solid, mark=diamond, color=cb-salmon-pink]
table[x={time}, y={primal}] {BCFW_A20_500_custom20-\dateconvex.txt};
\end{axis}
\end{tikzpicture}
}
\quad
\subfloat[][$s=100$]{
\centering
\begin{tikzpicture}[scale=1.0]
\begin{axis}[height=4.2cm,width=4.1cm, legend cell align={left},
ylabel style={yshift=-0.75em,font=\scriptsize}, 
xlabel style={yshift=0.50em,font=\scriptsize},
tick label style={font=\tiny},
scaled x ticks=false,
xtick={0,3000,6000,9000},
minor tick num=0,
xlabel=Iteration ($t$), 
ylabel=$f(x)-f(x^*)$,
grid,
xmin =0, xmax=10000, ymin=5e-3, ymax=1e1, ymode=log, mark repeat=18]
\addplot+[thick, mark=triangle, mark repeat=15, color=cb-burgundy] table[x={iter}, y={primal}] {BCFW_A20_100_full-\dateconvex.txt};
\addplot+[thick, mark=oplus, color=cb-brown]
table[x={iter}, y={primal}] {BCFW_A20_100_custom5-\dateconvex.txt};
\addplot+[thick, mark=square, color=cb-blue] table[x={iter}, y={primal}] {BCFW_A20_100_cyclic-\dateconvex.txt};
\addplot+[thick, mark=star, color=cb-lilac]
table[x={iter}, y={primal}] {BCFW_A20_100_custom10-\dateconvex.txt};
\addplot+[thick, mark=o, color=cb-green-sea] table[x={iter}, y={primal}] {BCFW_A20_100_stoc-\dateconvex.txt};
\addplot+[thick,solid, mark=otimes, mark repeat=30, color=cb-blue-light]
table[x={iter}, y={primal}] {BCFW_A20_100_ecyc20-\dateconvex.txt};
\addplot+[thick,solid, mark=diamond, color=cb-salmon-pink]
table[x={iter}, y={primal}] {BCFW_A20_100_custom20-\dateconvex.txt};
\end{axis}
\end{tikzpicture}
}
\subfloat[][$s=300$]{
\centering
\begin{tikzpicture}[scale=1.0]
\begin{axis}[height=4.2cm,width=4.1cm, legend cell align={left},
ylabel style={yshift=-0.75em,font=\scriptsize}, 
xlabel style={yshift=0.50em,font=\scriptsize},
tick label style={font=\tiny},
xlabel=Iteration ($t$), 
%ylabel=$f(x)-f(x^*)$,
scaled x ticks=false,
xtick={0,3000,6000,9000},
grid,
xmin =0, xmax=10000, ymin=1e-2, ymax=1e2, ymode=log, mark repeat=18]
\addplot+[thick, mark=triangle, mark repeat=15, color=cb-burgundy] table[x={iter}, y={primal}] {BCFW_A20_300_full-\dateconvex.txt};
\addplot+[thick, mark=oplus, color=cb-brown]
table[x={iter}, y={primal}] {BCFW_A20_300_custom5-\dateconvex.txt};
\addplot+[thick, mark=square, color=cb-blue] table[x={iter}, y={primal}] {BCFW_A20_300_cyclic-\dateconvex.txt};
\addplot+[thick, mark=star, color=cb-lilac]
table[x={iter}, y={primal}] {BCFW_A20_300_custom10-\dateconvex.txt};
\addplot+[thick, mark=o, color=cb-green-sea] table[x={iter}, y={primal}] {BCFW_A20_300_stoc-\dateconvex.txt};
\addplot+[thick,solid, mark=otimes, mark repeat=30, color=cb-blue-light]
table[x={iter}, y={primal}] {BCFW_A20_300_ecyc20-\dateconvex.txt};
\addplot+[thick,solid, mark=diamond, color=cb-salmon-pink]
table[x={iter}, y={primal}] {BCFW_A20_300_custom20-\dateconvex.txt};
\end{axis} \end{tikzpicture}
}
\subfloat[][$s=500$]{
\centering
\begin{tikzpicture}[scale=1.0]
\begin{axis}[height=4.2cm,width=4.1cm, legend cell align={left},
ylabel style={yshift=-0.75em,font=\scriptsize}, 
xlabel style={yshift=0.50em,font=\scriptsize},
tick label style={font=\tiny},
legend columns=2,
legend style={font=\tiny},
scaled x ticks=false,
xtick={0,3000,6000,9000},
xlabel=Iteration ($t$), 
%ylabel=$f(x)-f(x^*)$,
grid,
xmin =0, xmax=10000, ymin=4e-2, ymax=1e2, ymode=log, mark repeat=18]
\addplot+[thick, mark=triangle, mark repeat=15, color=cb-burgundy] table[x={iter}, y={primal}] {BCFW_A20_500_full-\dateconvex.txt};
\addplot+[thick, mark=oplus, color=cb-brown]
table[x={iter}, y={primal}] {BCFW_A20_500_custom5-\dateconvex.txt};
\addplot+[thick, mark=square, color=cb-blue] table[x={iter}, y={primal}] {BCFW_A20_500_cyclic-\dateconvex.txt};
\addplot+[thick, mark=star, color=cb-lilac]
table[x={iter}, y={primal}] {BCFW_A20_500_custom10-\dateconvex.txt};
\addplot+[thick, mark=o, color=cb-green-sea] table[x={iter}, y={primal}] {BCFW_A20_500_stoc-\dateconvex.txt};
\addplot+[thick,solid, mark=otimes, mark repeat=30, color=cb-blue-light]
table[x={iter}, y={primal}] {BCFW_A20_500_ecyc20-\dateconvex.txt};
\addplot+[thick,solid, mark=diamond, color=cb-salmon-pink]
table[x={iter}, y={primal}] {BCFW_A20_500_custom20-\dateconvex.txt};
\end{axis}
\end{tikzpicture}
}
\quad
\subfloat[][$s=100$]{
\centering
\begin{tikzpicture}[scale=1.0]
\begin{axis}[height=4.2cm,width=4.1cm, legend cell align={left},
ylabel style={yshift=-0.75em,font=\scriptsize}, 
xlabel style={yshift=0.50em,font=\scriptsize},
tick label style={font=\tiny},
scaled x ticks=false,
xtick={0,3000,6000,9000},
minor tick num=0,
xlabel=Spectr. LMOs, 
ylabel=$f(x)-f(x^*)$,
grid,
xmin =0, xmax=10000, ymin=5e-3, ymax=1e1, ymode=log, mark repeat=18]
\addplot+[thick, mark=triangle, mark repeat=15, color=cb-burgundy] table[x={lmo1}, y={primal}] {BCFW_A20_100_full-\dateconvex.txt};
\addplot+[thick, mark=oplus, color=cb-brown]
table[x={lmo1}, y={primal}] {BCFW_A20_100_custom5-\dateconvex.txt};
\addplot+[thick, mark=square, color=cb-blue] table[x={lmo1}, y={primal}] {BCFW_A20_100_cyclic-\dateconvex.txt};
\addplot+[thick, mark=star, color=cb-lilac]
table[x={lmo1}, y={primal}] {BCFW_A20_100_custom10-\dateconvex.txt};
\addplot+[thick, mark=o, color=cb-green-sea] table[x={lmo1}, y={primal}] {BCFW_A20_100_stoc-\dateconvex.txt};
\addplot+[thick,solid, mark=otimes, mark repeat=30, color=cb-blue-light]
table[x={lmo1}, y={primal}] {BCFW_A20_100_ecyc20-\dateconvex.txt};
\addplot+[thick,solid, mark=diamond, color=cb-salmon-pink]
table[x={lmo1}, y={primal}] {BCFW_A20_100_custom20-\dateconvex.txt};
\end{axis}
\end{tikzpicture}
}
\subfloat[][$s=300$]{
\centering
\begin{tikzpicture}[scale=1.0]
\begin{axis}[height=4.2cm,width=4.1cm, legend cell align={left},
ylabel style={yshift=-0.75em,font=\scriptsize}, 
xlabel style={yshift=0.50em,font=\scriptsize},
tick label style={font=\tiny},
xlabel=Spectr. LMOs, 
%ylabel=$f(x)-f(x^*)$,
scaled x ticks=false,
xtick={0,3000,6000,9000},
grid,
xmin =0, xmax=10000, ymin=1e-2, ymax=1e2, ymode=log, mark repeat=18]
\addplot+[thick, mark=triangle, mark repeat=15, color=cb-burgundy] table[x={lmo1}, y={primal}] {BCFW_A20_300_full-\dateconvex.txt};
\addplot+[thick, mark=oplus, color=cb-brown]
table[x={lmo1}, y={primal}] {BCFW_A20_300_custom5-\dateconvex.txt};
\addplot+[thick, mark=square, color=cb-blue] table[x={lmo1}, y={primal}] {BCFW_A20_300_cyclic-\dateconvex.txt};
\addplot+[thick, mark=star, color=cb-lilac]
table[x={lmo1}, y={primal}] {BCFW_A20_300_custom10-\dateconvex.txt};
\addplot+[thick, mark=o, color=cb-green-sea] table[x={lmo1}, y={primal}] {BCFW_A20_300_stoc-\dateconvex.txt};
\addplot+[thick,solid, mark=otimes, mark repeat=30, color=cb-blue-light]
table[x={lmo1}, y={primal}] {BCFW_A20_300_ecyc20-\dateconvex.txt};
\addplot+[thick,solid, mark=diamond, color=cb-salmon-pink]
table[x={lmo1}, y={primal}] {BCFW_A20_300_custom20-\dateconvex.txt};
\end{axis} \end{tikzpicture}
}
\subfloat[][$s=500$]{
\centering
\begin{tikzpicture}[scale=1.0]
\begin{axis}[height=4.2cm,width=4.1cm, legend cell align={left},
ylabel style={yshift=-0.75em,font=\scriptsize}, 
xlabel style={yshift=0.50em,font=\scriptsize},
tick label style={font=\tiny},
legend columns=2,
legend style={font=\tiny},
scaled x ticks=false,
xtick={0,3000,6000,9000},
xlabel=Spectr. LMOs, 
%ylabel=$f(x)-f(x^*)$,
grid,
xmin =0, xmax=10000, ymin=4e-2, ymax=1e2, ymode=log, mark repeat=18]
\addplot+[thick, mark=triangle, mark repeat=15, color=cb-burgundy] table[x={lmo1}, y={primal}] {BCFW_A20_500_full-\dateconvex.txt};
\addplot+[thick, mark=oplus, color=cb-brown]
table[x={lmo1}, y={primal}] {BCFW_A20_500_custom5-\dateconvex.txt};
\addplot+[thick, mark=square, color=cb-blue] table[x={lmo1}, y={primal}] {BCFW_A20_500_cyclic-\dateconvex.txt};
\addplot+[thick, mark=star, color=cb-lilac]
table[x={lmo1}, y={primal}] {BCFW_A20_500_custom10-\dateconvex.txt};
\addplot+[thick, mark=o, color=cb-green-sea] table[x={lmo1}, y={primal}] {BCFW_A20_500_stoc-\dateconvex.txt};
\addplot+[thick,solid, mark=otimes, mark repeat=30, color=cb-blue-light]
table[x={lmo1}, y={primal}] {BCFW_A20_500_ecyc20-\dateconvex.txt};
\addplot+[thick,solid, mark=diamond, color=cb-salmon-pink]
table[x={lmo1}, y={primal}] {BCFW_A20_500_custom20-\dateconvex.txt};
\end{axis}
\end{tikzpicture}
}
\end{center}
\caption{Results of intersection for
  a cube and spectraplex (Section~\ref{sec:ex1})
displaying the primal gap $f(\bx_t)-f(\bx^*)$ versus
time, iteration, and spectrahedral LMO count for problems with $s^2$
variables and block-activation strategies
\ref{act:1}--\ref{act:4} (E-cyclic uses $K=20$) and
\eqref{e:lazy}.}
\label{fig:ex1-shortstep}
\end{figure}

\begin{figure}
  \pgfplotslegendfromname{legend_adaptive}
  \begin{center}
    \subfloat{
      \centering
      \begin{tikzpicture}[scale=1.0]
      \begin{axis}[height=4.2cm,width=4.1cm, legend cell align={left},
      ylabel style={yshift=-0.75em,font=\scriptsize}, 
      xlabel style={yshift=0.50em,font=\scriptsize},
      tick label style={font=\tiny},
      legend columns=5,
      legend entries={Full-S,Cyclic-S,5-Lazy-S,10-Lazy-S,20-Lazy-S,Full-A,Cyclic-A,5-Lazy-A,10-Lazy-A,20-Lazy-A},
      legend style={font=\tiny, 
      at={(0.0,1.5)},
      anchor=north west,},
      legend to name={legend_adaptive},
      scaled x ticks=false,
      % xtick={0,5000,10000},
      %xlabel=Time (sec), 
      ylabel=$f(x)-f(x^*)$,
      grid,
      xmin =0, xmax=1300, ymin=1e-2, ymax=1e3, ymode=log, mark repeat=18]
      %actual plots short-step
      \addplot+[thick, mark=triangle, mark repeat=15, color=cb-burgundy] table[x={time}, y={primal}] {BCFW_A20_500_full-\dateconvex.txt};
      \addplot+[thick, mark=square, color=cb-blue] table[x={time}, y={primal}] {BCFW_A20_500_cyclic-\dateconvex.txt};

      % dummy plots for legend
      \addlegendimage{thick, mark=oplus, color=cb-brown}
      \addlegendimage{thick, mark=star, color=cb-lilac}
      \addlegendimage{thick, solid, mark=diamond, color=cb-salmon-pink}

      % actual plots adaptive
      \addplot+[thick, mark=o, color=cb-green-sea] table[x={time}, y={primal}] {BCFW_adpative_A20_500_full-\dateadaptive.txt};
      \addplot+[thick, solid, mark=otimes, mark repeat=30, color=cb-blue-light] table[x={time}, y={primal}] {BCFW_adpative_A20_500_cyclic-\dateadaptive.txt};

      % dummy plots for legend
      \addlegendimage{thick, mark=10-pointed star, color=cb-clay}
      \addlegendimage{thick, mark=|, color=cb-rose}
      \addlegendimage{thick, solid, mark=pentagon, color=cb-black}
      \end{axis}
      \end{tikzpicture}
    }
    \subfloat{
      \centering
      \begin{tikzpicture}[scale=1.0]
      \begin{axis}[height=4.2cm,width=4.1cm, legend cell align={left},
      ylabel style={yshift=-0.75em,font=\scriptsize}, 
      xlabel style={yshift=0.50em,font=\scriptsize},
      tick label style={font=\tiny},
      legend columns=2,
      legend style={font=\tiny},
      scaled x ticks=false,
      xtick={0,5000,10000},
      %xlabel=Iteration $t$, 
      grid,
      xmin =0, xmax=10000, ymin=1e-2, ymax=1e3, ymode=log, mark repeat=18]
      \addplot+[thick, mark=triangle, mark repeat=15, color=cb-burgundy] table[x={iter}, y={primal}] {BCFW_A20_500_full-\dateconvex.txt};
      \addplot+[thick, mark=o, color=cb-green-sea] table[x={iter}, y={primal}] {BCFW_adpative_A20_500_full-\dateadaptive.txt};
      \addplot+[thick, mark=square, color=cb-blue] table[x={iter}, y={primal}] {BCFW_A20_500_cyclic-\dateconvex.txt};
      \addplot+[thick,solid, mark=otimes, mark repeat=30, color=cb-blue-light]
      table[x={iter}, y={primal}] {BCFW_adpative_A20_500_cyclic-\dateadaptive.txt};
      \end{axis}
      \end{tikzpicture}
    }
    \subfloat{
      \centering
      \begin{tikzpicture}[scale=1.0]
      \begin{axis}[height=4.2cm,width=4.1cm, legend cell align={left},
      ylabel style={yshift=-0.75em,font=\scriptsize}, 
      xlabel style={yshift=0.50em,font=\scriptsize},
      tick label style={font=\tiny},
      legend columns=2,
      legend style={font=\tiny},
      scaled x ticks=false,
      xtick={0,5000,10000},
      %xlabel=Spectr. LMOs,
      grid,
      xmin =0, xmax=10000, ymin=1e-2, ymax=1e3, ymode=log, mark repeat=18]
      \addplot+[thick, mark=triangle, mark repeat=15, color=cb-burgundy] table[x={lmo1}, y={primal}] {BCFW_A20_500_full-\dateconvex.txt};
      \addplot+[thick, mark=o, color=cb-green-sea] table[x={lmo1}, y={primal}] {BCFW_adpative_A20_500_full-\dateadaptive.txt};
      \addplot+[thick, mark=square, color=cb-blue] table[x={lmo1}, y={primal}] {BCFW_A20_500_cyclic-\dateconvex.txt};
      \addplot+[thick,solid, mark=otimes, mark repeat=30, color=cb-blue-light]
      table[x={lmo1}, y={primal}] {BCFW_adpative_A20_500_cyclic-\dateadaptive.txt};
      \end{axis}
      \end{tikzpicture}
    }
    \quad
    \subfloat{
      \centering
      \begin{tikzpicture}[scale=1.0]
      \begin{axis}[height=4.2cm,width=4.1cm, legend cell align={left},
      ylabel style={yshift=-0.75em,font=\scriptsize}, 
      xlabel style={yshift=0.50em,font=\scriptsize},
      tick label style={font=\tiny},
      legend columns=2,
      legend style={font=\tiny, 
      at={(1.0,1.5)},
      anchor=north east,},
      scaled x ticks=false,
      xlabel=Time (sec), 
      ylabel=$f(x)-f(x^*)$,
      grid,
      xmin =0, xmax=500, ymin=1e-2, ymax=1e3, ymode=log, mark repeat=18]
      \addplot+[thick, mark=oplus, color=cb-brown]
      table[x={time}, y={primal}] {BCFW_A20_500_custom5-\dateconvex.txt};
      \addplot+[thick, mark=10-pointed star, color=cb-clay]
      table[x={time}, y={primal}] {BCFW_adpative_A20_500_custom5-\dateadaptive.txt};
      \addplot+[thick, mark=star, color=cb-lilac]
      table[x={time}, y={primal}] {BCFW_A20_500_custom10-\dateconvex.txt};
      \addplot+[thick, mark=|, color=cb-rose]
      table[x={time}, y={primal}] {BCFW_adpative_A20_500_custom10-\dateadaptive.txt};
      \addplot+[thick,solid, mark=diamond, color=cb-salmon-pink]
      table[x={time}, y={primal}] {BCFW_A20_500_custom20-\dateconvex.txt};
      \addplot+[thick,solid, mark=pentagon, color=cb-black]
      table[x={time}, y={primal}] {BCFW_adpative_A20_500_custom20-\dateadaptive.txt};
      \end{axis}
      \end{tikzpicture}
    }
    \subfloat{
      \centering
      \begin{tikzpicture}[scale=1.0]
      \begin{axis}[height=4.2cm,width=4.1cm, legend cell align={left},
      ylabel style={yshift=-0.75em,font=\scriptsize}, 
      xlabel style={yshift=0.50em,font=\scriptsize},
      tick label style={font=\tiny},
      legend columns=2,
      legend style={font=\tiny},
      scaled x ticks=false,
      xtick={0,5000,10000},
      xlabel=Iteration ($t$), 
      %ylabel=$f(x)-f(x^*)$,
      grid,
      xmin =0, xmax=10000, ymin=1e-2, ymax=1e3, ymode=log, mark repeat=18]
      \addplot+[thick, mark=oplus, color=cb-brown]
      table[x={iter}, y={primal}] {BCFW_A20_500_custom5-\dateconvex.txt};
      \addplot+[thick, mark=10-pointed star, color=cb-clay]
      table[x={iter}, y={primal}] {BCFW_adpative_A20_500_custom5-\dateadaptive.txt};
      \addplot+[thick, mark=star, color=cb-lilac]
      table[x={iter}, y={primal}] {BCFW_A20_500_custom10-\dateconvex.txt};
      \addplot+[thick, mark=|, color=cb-rose]
      table[x={iter}, y={primal}] {BCFW_adpative_A20_500_custom10-\dateadaptive.txt};
      \addplot+[thick,solid, mark=diamond, color=cb-salmon-pink]
      table[x={iter}, y={primal}] {BCFW_A20_500_custom20-\dateconvex.txt};
      \addplot+[thick,solid, mark=pentagon, color=cb-black]
      table[x={iter}, y={primal}] {BCFW_adpative_A20_500_custom20-\dateadaptive.txt};
      \end{axis}
      \end{tikzpicture}
    }
    \subfloat{
      \centering
      \begin{tikzpicture}[scale=1.0]
      \begin{axis}[height=4.2cm,width=4.1cm, legend cell align={left},
      ylabel style={yshift=-0.75em,font=\scriptsize}, 
      xlabel style={yshift=0.50em,font=\scriptsize},
      tick label style={font=\tiny},
      legend columns=2,
      legend style={font=\tiny},
      scaled x ticks=false,
      xtick={0,1000,2000},
      xlabel=Spectr. LMOs, 
      %ylabel=$f(x)-f(x^*)$,
      grid,
      xmin =0, xmax=2000, ymin=1e-2, ymax=1e3, ymode=log, mark repeat=18]
      \addplot+[thick, mark=oplus, color=cb-brown]
      table[x={lmo1}, y={primal}] {BCFW_A20_500_custom5-\dateconvex.txt};
      \addplot+[thick, mark=10-pointed star, color=cb-clay]
      table[x={lmo1}, y={primal}] {BCFW_adpative_A20_500_custom5-\dateadaptive.txt};
      \addplot+[thick, mark=star, color=cb-lilac]
      table[x={lmo1}, y={primal}] {BCFW_A20_500_custom10-\dateconvex.txt};
      \addplot+[thick, mark=|, color=cb-rose]
      table[x={lmo1}, y={primal}] {BCFW_adpative_A20_500_custom10-\dateadaptive.txt};
      \addplot+[thick,solid, mark=diamond, color=cb-salmon-pink]
      table[x={lmo1}, y={primal}] {BCFW_A20_500_custom20-\dateconvex.txt};
      \addplot+[thick,solid, mark=pentagon, color=cb-black]
      table[x={lmo1}, y={primal}] {BCFW_adpative_A20_500_custom20-\dateadaptive.txt};
      \end{axis}
      \end{tikzpicture}
    }
  \end{center}
  \caption{Results of intersection for a cube and spectraplex for $s=500$
    (Section~\ref{sec:ex1}) comparing adaptive step sizes
    (Algorithm~\ref{a:abcg}) labelled with suffix ``-A'' to short-step
    sizes (Algorithm~\ref{a:bcg}) labelled with suffix ``-S''. The first row compares the performance of the the Full and Cyclic activation strategies, the second row compares the lazy activation strategies.}
  \label{fig:ex1-comparison}
\end{figure}

\subsection{Experiment 2: Difference of convex quadratics}
\label{sec:ex2}

The goal is to minimize a difference of convex quadratic
functions of two collated matrices in $\mathbb{R}^{s\times s}$,
where the submatrices are constrained to an $\ell_{\infty}$ ball and a
nuclear-norm ball respectively. In order to examine the performance
of Algorithm~\ref{a:bcg} when the number of components is large,
we split the $\ell_{\infty}$ constraint into $s$ separate
constraints. Hence, we set $C_1=\cdots=C_s\coloneqq
\menge{x\in\mathbb{R}^s}{\|x\|_{\infty}\leq 1}$, and
$C_{s+1}=\menge{x\in\mathbb{R}^{s\times
s}}{\|x\|_{\text{nuc}}\leq 1}$. For $\bx\in\CCX$ we use $[\bx]$
to denote the collated $2s\times s$ matrix of its components.
For each problem instance, the kernel $A,B$ of each quadratic is
generated by projecting a matrix with random normal entries of mean
$0$ and standard deviation $1$ onto the set of positive
semidefinite matrices. Altogether, we seek to solve the following
difference-of-convex problem involving the Frobenius inner product
\begin{equation}
\label{e:ex2}
\minimize_{\bx\in C_1\times\cdots\times C_{s+1}}
\frac{1}{2}
\Big(\Big\langle{[\bx]}\;\Big|\:{[\bx] A}\Big\rangle
-
\Big\langle{[\bx]}\;\Big|\:{[\bx] B}\Big\rangle\Big).
\end{equation}
Note the objective function of \eqref{e:ex2} is
smooth and nonseparable.
In the experiments we used the
Froebenous norm of \(A-B\)
as smoothness constant \(L_{f}\).
%as it is easier to compute than the best smoothness constant
%(spectral norm of \(A-B\)).
For each instance of
\eqref{e:ex2}, we verify that $A-B$ is indefinite, hence
the objective is also neither convex nor concave.
Since the $\ell_{\infty}$ LMO
is far cheaper than the nuclear norm ball LMO \cite{Comb21LMO},
similarly to 
Section~\ref{sec:ex1}, we consider a family of customized
activation strategies that delay evaluating the most expensive
operator $\lmo_{s+1}$. In addition, on the ``lazy'' iterations
involving only the LMOs of the $\ell_{\infty}$ norm ball, we
perform a parallel update involving a random subset of $I\setminus
\{s+1\}$ of size $p$:
\begin{equation}
\label{e:bx2}
(\forall t\in\NN)\quad I_t=\begin{cases}
I &\text{if}\;\;t\equiv 0 \pmod q \\
\{i_1,\ldots,i_p\}\subset I\setminus \{s+1\} &\text{otherwise.}
\end{cases}
\end{equation}

Averaged results from $20$ instances of \eqref{e:ex2} are shown in
Figure~\ref{fig:22}. Since the problem is
nonconvex, we plot the
minimal F-W gap observed (see Theorem~\ref{t:ncv-conv}). Since full
F-W gaps are typically
unavailable in BCFW (only partial gaps for
the activated blocks $(G_i)_{i\in I_t}$ are computed), iterates
were stored during the run of Algorithm~\ref{a:bcg} and full F-W
gaps were computed post-hoc. 

Similarly to Section~\ref{sec:ex1}, new selection strategies
allowed by Assumption~\ref{a:1} can yield similar per-iteration
performance to that of full-activation Frank-Wolfe; furthermore,
since the iterations frequently involve the cheaper LMOs, this can
yield faster convergence in wall-clock time. However, these results
also demonstrate that, if the number of activated components in
$I_t$ is too small and the $(s+1)$st component is activated too
infrequently, results may worsen; this is reflected in the cyclic,
$P$-cyclic, E-cyclic ($K=2s$), and $(p,q)=(2,20)$ results. The
on-par performance with full-activation Frank-Wolfe suggests that,
particularly if gradients are expensive relative to LMOs (as is the
case here, but not in Experiment 1), selecting a larger number of
(cheap) LMOs to update may yield a beneficial balance of
progress-per-iteration and wallclock time.

\begin{figure}[htb]
\pgfplotslegendfromname{legend_nonconv_ex2}
\begin{center}
\subfloat[][$s=100$]{
\begin{tikzpicture}[scale=1.0]
\begin{axis}[height=4.2cm,width=4.1cm, legend cell align={left},
label style={font=\tiny},
ylabel style={yshift=-0.75em,font=\scriptsize}, 
xlabel style={yshift=0.50em,font=\scriptsize},
tick label style={font=\tiny},
legend columns=2,
%legend entries={$(2,20)$\\
%                Full\\
%                $(10,10)$\\
%                Cyclic\\
%                $(\frac{n}{2},5)$\\
%                P-Cyclic\\
%               $(n,2)$\\
%    },
legend style={font=\tiny},
xlabel=Time (sec), 
ylabel=Min. F-W gap,
grid,
xmin =0, xmax=20, ymin=1e-1, ymax=1e6, ymode=log, mark repeat=18]
\addplot+[thick, mark=diamond, solid, color=cb-salmon-pink]
table[x={time}, y={dmin}] {DC_A20_100_custom20-\datenonconvex.txt};
\addplot+[thick, mark=triangle, mark repeat=15, color=cb-burgundy] table[x={time}, y={dmin}] {DC_A20_100_full-\datenonconvex.txt};
\addplot+[thick, mark=oplus, color=cb-brown] table[x={time},
y={dmin}] {DC_A20_100_custom10-\datenonconvex.txt};
\addplot+[thick, mark=square, color=cb-blue] table[x={time}, y={dmin}] {DC_A20_100_cyclic-\datenonconvex.txt};
\addplot+[thick, mark=star,solid, mark repeat=20, color=cb-lilac]
table[x={time}, y={dmin}] {DC_A20_100_custom5-\datenonconvex.txt};
\addplot+[thick, mark=o,solid, color=cb-green-sea] table[x={time}, y={dmin}] {DC_A20_100_stoc-\datenonconvex.txt};
\addplot+[thick, mark=otimes*,solid, mark repeat=20,
color=cb-green-lime] table[x={time}, y={dmin}] {DC_A20_100_custom2-\datenonconvex.txt};
\addplot+[thick,solid, mark=otimes, mark repeat=20, color=cb-blue-light]
table[x={time}, y={dmin}] {DC_A20_100_ecyc200-\datenonconvex.txt};
\end{axis}
\end{tikzpicture}
}
\subfloat[][$s=300$]{
\begin{tikzpicture}[scale=1.0]
\begin{axis}[height=4.2cm,width=4.1cm, legend cell align={left},
label style={font=\tiny},
ylabel style={yshift=-0.75em,font=\scriptsize}, 
xlabel style={yshift=0.50em,font=\scriptsize},
tick label style={font=\tiny},
legend columns=2,
%legend entries={$(2,20)$\\
%                Full\\
%                $(10,10)$\\
%                Cyclic\\
%                $(\frac{n}{2},5)$\\
%                P-Cyclic\\
%               $(n,2)$\\
%    },
legend style={font=\tiny},
xlabel=Time (sec), 
%ylabel=Min. F-W gap,
grid,
xmin =0, xmax=100, ymin=1e1, ymax=1e7, ymode=log, mark repeat=18]
\addplot+[thick, mark=diamond, solid, color=cb-salmon-pink]
table[x={time}, y={dmin}] {DC_A20_300_custom20-\datenonconvex.txt};
\addplot+[thick, mark=triangle, mark repeat=15, color=cb-burgundy] table[x={time}, y={dmin}] {DC_A20_300_full-\datenonconvex.txt};
\addplot+[thick, mark=oplus, color=cb-brown] table[x={time},
y={dmin}] {DC_A20_300_custom10-\datenonconvex.txt};
\addplot+[thick, mark=square, mark repeat=8,color=cb-blue] table[x={time}, y={dmin}] {DC_A20_300_cyclic-\datenonconvex.txt};
\addplot+[thick, mark=star,solid, mark repeat=20, color=cb-lilac]
table[x={time}, y={dmin}] {DC_A20_300_custom5-\datenonconvex.txt};
\addplot+[thick,mark repeat=6, mark=o,solid, color=cb-green-sea] table[x={time}, y={dmin}] {DC_A20_300_stoc-\datenonconvex.txt};
\addplot+[thick, mark=otimes*,solid, mark repeat=20,
color=cb-green-lime] table[x={time}, y={dmin}] {DC_A20_300_custom2-\datenonconvex.txt};
\addplot+[thick,solid, mark=otimes, mark repeat=20, color=cb-blue-light]
table[x={time}, y={dmin}] {DC_A20_300_ecyc600-\datenonconvex.txt};
\end{axis}
\end{tikzpicture}
}
\subfloat[][$s=500$]{
\begin{tikzpicture}[scale=1.0]
\begin{axis}[height=4.2cm,width=4.1cm, legend cell align={left},
label style={font=\tiny},
ylabel style={yshift=-0.75em,font=\scriptsize}, 
xlabel style={yshift=0.50em,font=\scriptsize},
tick label style={font=\tiny},
scaled x ticks=false,
legend columns=4,
legend entries={$(2,20)$\\
                $(10,10)$\\
                $(\frac{s}{2},5)$\\
               $(s,2)$\\
               Full\\
               Cyclic\\
               P-Cyclic\\
               E-Cyclic\\
    },
legend to name={legend_nonconv_ex2},
legend style={
    font=\tiny,
    at={(1.0,1.6)},
    anchor=north east,
    },
xlabel=Time (sec),
%ylabel=Min. F-W gap,
grid,
xmin =0, xmax=240, ymin=1e2, ymax=5e7, ymode=log, mark repeat=18]
\addplot+[thick, mark=diamond, solid, color=cb-salmon-pink]
table[x={time}, y={dmin}] {DC_A20_500_custom20-\datenonconvex.txt};
\addplot+[thick, mark=oplus, color=cb-brown] table[x={time},
y={dmin}] {DC_A20_500_custom10-\datenonconvex.txt};
\addplot+[thick, mark=star,solid, mark repeat=20, color=cb-lilac]
table[x={time}, y={dmin}] {DC_A20_500_custom5-\datenonconvex.txt};
\addplot+[thick, mark=otimes*,solid, mark repeat=20,
color=cb-green-lime] table[x={time}, y={dmin}] {DC_A20_500_custom2-\datenonconvex.txt};
\addplot+[thick, mark=triangle, mark repeat=15, color=cb-burgundy] table[x={time}, y={dmin}] {DC_A20_500_full-\datenonconvex.txt};
\addplot+[thick, mark=square,mark repeat=8, color=cb-blue] table[x={time}, y={dmin}] {DC_A20_500_cyclic-\datenonconvex.txt};
\addplot+[thick,mark repeat=6, mark=o,solid, color=cb-green-sea] table[x={time}, y={dmin}] {DC_A20_500_stoc-\datenonconvex.txt};
\addplot+[thick,solid, mark=otimes, mark repeat=20, color=cb-blue-light]
table[x={time}, y={dmin}] {DC_A20_500_ecyc1000-\datenonconvex.txt};
\end{axis}
\end{tikzpicture}
}

\subfloat[][$s=100$]{
\begin{tikzpicture}[scale=1.0]
\begin{axis}[height=4.2cm,width=4.1cm, legend cell align={left},
label style={font=\tiny},
ylabel style={yshift=-0.75em,font=\scriptsize}, 
xlabel style={yshift=0.50em,font=\scriptsize},
tick label style={font=\tiny},
scaled x ticks=false,
xtick={0,3000,6000,9000},
legend columns=2,
%legend entries={$(2,20)$\\
%                Full\\
%                $(10,10)$\\
%                Cyclic\\
%                $(\frac{n}{2},5)$\\
%                P-Cyclic\\
%               $(n,2)$\\
%    },
legend style={font=\tiny},
xlabel=Iteration ($t$),
ylabel=Min. F-W gap,
grid,
xmin =0, xmax=10000, ymin=1e-1, ymax=1e6, ymode=log, mark repeat=18]
\addplot+[thick, mark=diamond, solid, color=cb-salmon-pink]
table[x={iter}, y={dmin}] {DC_A20_100_custom20-\datenonconvex.txt};
\addplot+[thick, mark=triangle, mark repeat=15, color=cb-burgundy] table[x={iter}, y={dmin}] {DC_A20_100_full-\datenonconvex.txt};
\addplot+[thick, mark=oplus, color=cb-brown] table[x={iter},
y={dmin}] {DC_A20_100_custom10-\datenonconvex.txt};
\addplot+[thick, mark=square, color=cb-blue] table[x={iter}, y={dmin}] {DC_A20_100_cyclic-\datenonconvex.txt};
\addplot+[thick, mark=star,solid, mark repeat=20, color=cb-lilac]
table[x={iter}, y={dmin}] {DC_A20_100_custom5-\datenonconvex.txt};
\addplot+[thick, mark=o,solid, color=cb-green-sea] table[x={iter}, y={dmin}] {DC_A20_100_stoc-\datenonconvex.txt};
\addplot+[thick, mark=otimes*,solid, mark repeat=20,
color=cb-green-lime] table[x={iter}, y={dmin}] {DC_A20_100_custom2-\datenonconvex.txt};
\addplot+[thick,solid, mark=otimes, mark repeat=20, color=cb-blue-light]
table[x={iter}, y={dmin}] {DC_A20_100_ecyc200-\datenonconvex.txt};
\end{axis}
\end{tikzpicture}
}
\subfloat[][$s=300$]{
\begin{tikzpicture}[scale=1.0]
\begin{axis}[height=4.2cm,width=4.1cm, legend cell align={left},
label style={font=\tiny},
ylabel style={yshift=-0.75em,font=\scriptsize}, 
xlabel style={yshift=0.50em,font=\scriptsize},
tick label style={font=\tiny},
scaled x ticks=false,
xtick={0,3000,6000,9000},
legend columns=2,
%legend entries={$(2,20)$\\
%                Full\\
%                $(10,10)$\\
%                Cyclic\\
%                $(\frac{n}{2},5)$\\
%                P-Cyclic\\
%               $(n,2)$\\
%    },
legend style={font=\tiny},
xlabel=Iteration ($t$),
%ylabel=Min. F-W gap,
grid,
xmin =0, xmax=10000, ymin=1e1, ymax=1e7, ymode=log, mark repeat=18]
\addplot+[thick, mark=diamond, solid, color=cb-salmon-pink]
table[x={iter}, y={dmin}] {DC_A20_300_custom20-\datenonconvex.txt};
\addplot+[thick, mark=triangle, mark repeat=15, color=cb-burgundy] table[x={iter}, y={dmin}] {DC_A20_300_full-\datenonconvex.txt};
\addplot+[thick, mark=oplus, color=cb-brown] table[x={iter},
y={dmin}] {DC_A20_300_custom10-\datenonconvex.txt};
\addplot+[thick, mark=square, mark repeat=8,color=cb-blue] table[x={iter}, y={dmin}] {DC_A20_300_cyclic-\datenonconvex.txt};
\addplot+[thick, mark=star,solid, mark repeat=20, color=cb-lilac]
table[x={iter}, y={dmin}] {DC_A20_300_custom5-\datenonconvex.txt};
\addplot+[thick,mark repeat=10, mark=o,solid, color=cb-green-sea] table[x={iter}, y={dmin}] {DC_A20_300_stoc-\datenonconvex.txt};
\addplot+[thick, mark=otimes*,solid, mark repeat=20,
color=cb-green-lime] table[x={iter}, y={dmin}] {DC_A20_300_custom2-\datenonconvex.txt};
\addplot+[thick,solid, mark=otimes, mark repeat=20, color=cb-blue-light]
table[x={iter}, y={dmin}] {DC_A20_300_ecyc600-\datenonconvex.txt};
\end{axis}
\end{tikzpicture}
}
\subfloat[][$s=500$]{
\begin{tikzpicture}[scale=1.0]
\begin{axis}[height=4.2cm,width=4.1cm, legend cell align={left},
label style={font=\tiny},
ylabel style={yshift=-0.75em,font=\scriptsize}, 
xlabel style={yshift=0.50em,font=\scriptsize},
tick label style={font=\tiny},
scaled x ticks=false,
xtick={0,3000,6000,9000},
legend columns=2,
%legend entries={$(2,20)$\\
%                Full\\
%                $(10,10)$\\
%                Cyclic\\
%                $(\frac{n}{2},5)$\\
%                P-Cyclic\\
%               $(n,2)$\\
%    },
legend style={font=\tiny},
xlabel=Iteration ($t$),
%ylabel=Min. F-W gap,
grid,
xmin =0, xmax=10000, ymin=1e2, ymax=5e7, ymode=log, mark repeat=18]
\addplot+[thick, mark=diamond, solid, color=cb-salmon-pink]
table[x={iter}, y={dmin}] {DC_A20_500_custom20-\datenonconvex.txt};
\addplot+[thick, mark=triangle, mark repeat=15, color=cb-burgundy] table[x={iter}, y={dmin}] {DC_A20_500_full-\datenonconvex.txt};
\addplot+[thick, mark=oplus, color=cb-brown] table[x={iter},
y={dmin}] {DC_A20_500_custom10-\datenonconvex.txt};
\addplot+[thick, mark=square, mark repeat=8,color=cb-blue] table[x={iter}, y={dmin}] {DC_A20_500_cyclic-\datenonconvex.txt};
\addplot+[thick, mark=star,solid, mark repeat=20, color=cb-lilac]
table[x={iter}, y={dmin}] {DC_A20_500_custom5-\datenonconvex.txt};
\addplot+[thick,mark repeat=10, mark=o,solid, color=cb-green-sea] table[x={iter}, y={dmin}] {DC_A20_500_stoc-\datenonconvex.txt};
\addplot+[thick, mark=otimes*,solid, mark repeat=20,
color=cb-green-lime] table[x={iter}, y={dmin}] {DC_A20_500_custom2-\datenonconvex.txt};
\addplot+[thick,solid, mark=otimes, mark repeat=20, color=cb-blue-light]
table[x={iter}, y={dmin}] {DC_A20_500_ecyc1000-\datenonconvex.txt};
\end{axis}
\end{tikzpicture}
}
\quad
\subfloat[][$s=100$]{
\begin{tikzpicture}[scale=1.0]
\begin{axis}[height=4.2cm,width=4.1cm, legend cell align={left},
label style={font=\tiny},
ylabel style={yshift=-0.75em,font=\scriptsize}, 
xlabel style={yshift=0.50em,font=\scriptsize},
tick label style={font=\tiny},
scaled x ticks=false,
xtick={0,2000,4000},
legend columns=2,
%legend entries={$(2,20)$\\
%                Full\\
%                $(10,10)$\\
%                Cyclic\\
%                $(\frac{n}{2},5)$\\
%                P-Cyclic\\
%               $(n,2)$\\
%    },
legend style={font=\tiny},
xlabel=Nucl. LMO,
ylabel=Min. F-W gap,
grid,
xmin =0, xmax=5000, ymin=1e-1, ymax=1e6, ymode=log, mark repeat=18]
\addplot+[thick, mark=diamond, solid, color=cb-salmon-pink]
table[x={lmo2}, y={dmin}] {DC_A20_100_custom20-\datenonconvex.txt};
\addplot+[thick, mark=triangle, mark repeat=15, color=cb-burgundy] table[x={lmo2}, y={dmin}] {DC_A20_100_full-\datenonconvex.txt};
\addplot+[thick, mark=oplus, color=cb-brown] table[x={lmo2},
y={dmin}] {DC_A20_100_custom10-\datenonconvex.txt};
\addplot+[thick, mark=square, color=cb-blue] table[x={lmo2}, y={dmin}] {DC_A20_100_cyclic-\datenonconvex.txt};
\addplot+[thick, mark=star,solid, mark repeat=20, color=cb-lilac]
table[x={lmo2}, y={dmin}] {DC_A20_100_custom5-\datenonconvex.txt};
\addplot+[thick, mark=o,solid, color=cb-green-sea] table[x={lmo2}, y={dmin}] {DC_A20_100_stoc-\datenonconvex.txt};
\addplot+[thick, mark=otimes*,solid, mark repeat=20,
color=cb-green-lime] table[x={lmo2}, y={dmin}] {DC_A20_100_custom2-\datenonconvex.txt};
\addplot+[thick,solid, mark=otimes, mark repeat=20, color=cb-blue-light]
table[x={lmo2}, y={dmin}] {DC_A20_100_ecyc200-\datenonconvex.txt};
\end{axis}
\end{tikzpicture}
}
\subfloat[][$s=300$]{
\begin{tikzpicture}[scale=1.0]
\begin{axis}[height=4.2cm,width=4.1cm, legend cell align={left},
label style={font=\tiny},
ylabel style={yshift=-0.75em,font=\scriptsize}, 
xlabel style={yshift=0.50em,font=\scriptsize},
tick label style={font=\tiny},
scaled x ticks=false,
xtick={0,2000,4000},
legend columns=2,
%legend entries={$(2,20)$\\
%                Full\\
%                $(10,10)$\\
%                Cyclic\\
%                $(\frac{n}{2},5)$\\
%                P-Cyclic\\
%               $(n,2)$\\
%    },
legend style={font=\tiny},
xlabel=Nucl. LMOs,
%ylabel=Min. F-W gap,
grid,
xmin =0, xmax=5000, ymin=1e1, ymax=1e7, ymode=log, mark repeat=18]
\addplot+[thick, mark=diamond, solid, color=cb-salmon-pink]
table[x={lmo2}, y={dmin}] {DC_A20_300_custom20-\datenonconvex.txt};
\addplot+[thick, mark=triangle, mark repeat=15, color=cb-burgundy] table[x={lmo2}, y={dmin}] {DC_A20_300_full-\datenonconvex.txt};
\addplot+[thick, mark=oplus, color=cb-brown] table[x={lmo2},
y={dmin}] {DC_A20_300_custom10-\datenonconvex.txt};
\addplot+[thick, mark=square, mark repeat=12, color=cb-blue] table[x={lmo2}, y={dmin}] {DC_A20_300_cyclic-\datenonconvex.txt};
\addplot+[thick, mark=star,solid, mark repeat=20, color=cb-lilac]
table[x={lmo2}, y={dmin}] {DC_A20_300_custom5-\datenonconvex.txt};
\addplot+[thick,mark repeat=10, mark=o,solid, color=cb-green-sea] table[x={lmo2}, y={dmin}] {DC_A20_300_stoc-\datenonconvex.txt};
\addplot+[thick, mark=otimes*,solid, mark repeat=20,
color=cb-green-lime] table[x={lmo2}, y={dmin}] {DC_A20_300_custom2-\datenonconvex.txt};
\addplot+[thick,solid, mark=otimes, mark repeat=20, color=cb-blue-light]
table[x={lmo2}, y={dmin}] {DC_A20_300_ecyc600-\datenonconvex.txt};
\end{axis}
\end{tikzpicture}
}
\subfloat[][$s=500$]{
\begin{tikzpicture}[scale=1.0]
\begin{axis}[height=4.2cm,width=4.1cm, legend cell align={left},
label style={font=\tiny},
ylabel style={yshift=-0.75em,font=\scriptsize}, 
xlabel style={yshift=0.50em,font=\scriptsize},
tick label style={font=\tiny},
scaled x ticks=false,
xtick={0,2000,4000},
legend columns=2,
%legend entries={$(2,20)$\\
%                Full\\
%                $(10,10)$\\
%                Cyclic\\
%                $(\frac{n}{2},5)$\\
%                P-Cyclic\\
%               $(n,2)$\\
%    },
legend style={font=\tiny},
xlabel=Nucl. LMOs,
%ylabel=Min. F-W gap,
grid,
xmin =0, xmax=5000, ymin=1e2, ymax=5e7, ymode=log, mark repeat=18]
\addplot+[thick, mark=diamond, solid, color=cb-salmon-pink]
table[x={lmo2}, y={dmin}] {DC_A20_500_custom20-\datenonconvex.txt};
\addplot+[thick, mark=triangle, mark repeat=15, color=cb-burgundy] table[x={lmo2}, y={dmin}] {DC_A20_500_full-\datenonconvex.txt};
\addplot+[thick, mark=oplus, color=cb-brown] table[x={lmo2},
y={dmin}] {DC_A20_500_custom10-\datenonconvex.txt};
\addplot+[thick, mark=square,mark repeat=8, color=cb-blue] table[x={lmo2}, y={dmin}] {DC_A20_500_cyclic-\datenonconvex.txt};
\addplot+[thick, mark=star,solid, mark repeat=20, color=cb-lilac]
table[x={lmo2}, y={dmin}] {DC_A20_500_custom5-\datenonconvex.txt};
\addplot+[thick,mark repeat=6, mark=o,solid, color=cb-green-sea] table[x={lmo2}, y={dmin}] {DC_A20_500_stoc-\datenonconvex.txt};
\addplot+[thick, mark=otimes*,solid, mark repeat=20,
color=cb-green-lime] table[x={lmo2}, y={dmin}] {DC_A20_500_custom2-\datenonconvex.txt};
\addplot+[thick,solid, mark=otimes, mark repeat=20, color=cb-blue-light]
table[x={lmo2}, y={dmin}] {DC_A20_500_ecyc1000-\datenonconvex.txt};
\end{axis}
\end{tikzpicture}
}
\end{center}
\caption{
Results for difference of quadratic functions (Section~\ref{sec:ex2})
displaying the minimal F-W gap versus time,
iteration, and nuclear norm LMO count for problems with
$s^2$ variables and block-activation strategies
\ref{act:1}--\ref{act:3} and \eqref{e:bx2}; $K=2s$ for E-cyclic
activation; for $(p,q)$ activation,
\(p\) is the number of ``cheaper'' coordinates activated per iteration,
and \(q\) is the frequency of full activation.
}
\label{fig:22}
\end{figure}

\subsection{Experiment 3: Structural SVM training}
\label{sec:ex3}

In the final experiment, we consider the structural SVM training problem for sequence labeling from \cite{Laco13} and analyze the OCR dataset \cite{Taskar03}, which contains $6251$ instances with $4028$ features.
This task is solved via a dual reformulation,
detailed in \cite{Laco13}, resulting in a problem of the form
\eqref{e:p}.
The block-coordinate Frank-Wolfe (BCFW) method in \cite{Laco13}
uses a permuted cyclic update strategy satisfying
Assumption~\ref{a:1}, but relies on componentwise optimal step
sizes on the dual problem, which are not guaranteed to converge under parallel updates (see Example~\ref{ex:nonconv}).

Computing the Lipschitz smoothness constant for this dual problem is infeasible, despite its quadratic objective, because the immense size of the the Hessian, which correpsonds to the number of potential outputs.
For the OCR dataset, with maximum sequence length $14$ and $36$ possible labels per position, there are $36^{14} \approx 6 \cdot 10^{21}$ potential outputs.
Consequently, BCFW with short-step updates must approximate the
smoothness constant, making this a natural use case for Algorithm
\ref{a:abcg}.

As computing the dual gradient involves the same large matrix and is thus intractable, we rely on
the smoothness inequality \eqref{e:2} rather than the interpolability condition in \eqref{eq:smooth3} (see Remark~\ref{r:smooth}).
While this still requires maintaining dual variables, they are highly sparse due to the structure of the LMO over the unit simplex, so the additional storage is manageable.
In contrast, the primal-dual method in \cite{Laco13} with exact line search only stores primal variables, avoiding this overhead.

Figure~\ref{fig:ex3} compares Algorithm \ref{a:abcg} with P-Cyclic
block updates of size $1$, $5$, and $10$ to the original method in
\cite{Laco13}, which uses both P-Cyclic and uniformly random
singleton updates.
All methods achieve similar primal progress per epoch and time.
The adaptive algorithm is slightly slower due to dual-variable maintenance and extra computations for smoothness estimation.
Updates with $5$ or $10$ blocks outperform singleton updates in
iteration count.

In conclusion, Algorithm~\ref{a:abcg} is a competitive alternative to the original BCFW method for structural SVM training.
It does not require knowledge of problem parameters like the smoothness constant, supports parallel updates, and can be applied to objectives where exact line search is infeasible or too costly.

\begin{figure}

\pgfplotslegendfromname{legend_bcfwstruct}
\begin{center}
\subfloat{
  \centering
  \begin{tikzpicture}[scale=1.0]
    \begin{axis}[height=4.2cm,width=4.1cm, legend cell align={left},
      ylabel style={yshift=-0.75em,font=\scriptsize}, 
      xlabel style={yshift=0.50em,font=\scriptsize},
      tick label style={font=\tiny},
      scaled x ticks=false,
      xlabel=Iterations, 
      ylabel=SVM-Primal,
      xtick={0, 40000, 80000},
      grid,
      xmin=0, xmax=100000, ymin=0.15, ymax=1, ymode=log, mark repeat=50]
      \addplot+[thick, mark=square, color=cb-blue] table[x={iter}, y={primal}] {progress_1_150_Adaptive_perm_20.txt};
      \addplot+[thick, mark=10-pointed star, color=cb-clay] table[x={iter}, y={primal}] {progress_5_150_Adaptive_perm_20.txt};
      \addplot+[thick, mark=oplus, color=cb-rose] table[x={iter}, y={primal}] {progress_10_150_Adaptive_perm_20.txt};
      \addplot+[thick, mark=triangle, mark repeat=15, color=cb-burgundy] table[x={iter}, y={primal}] {progress_1_150_Optimal_perm_20.txt};
      \addplot+[thick, mark=pentagon, color=cb-green-sea] table[x={iter}, y={primal}] {progress_1_150_Optimal_uniform_20.txt};
    \end{axis}
  \end{tikzpicture}
}
\subfloat{
  \centering
  \begin{tikzpicture}[scale=1.0]
    \begin{axis}[height=4.2cm,width=4.1cm, legend cell align={left},
      ylabel style={yshift=-0.75em,font=\scriptsize}, 
      xlabel style={yshift=0.50em,font=\scriptsize},
      tick label style={font=\tiny},
      scaled x ticks=false,
      xlabel=Epochs,
      grid,
      xmin=0, xmax=160, ymin=0.15, ymax=1, ymode=log, mark repeat=50]
      \addplot+[thick, mark=square, color=cb-blue] table[x={eff_pass}, y={primal}] {progress_1_150_Adaptive_perm_20.txt};
      \addplot+[thick, mark=10-pointed star, color=cb-clay] table[x={eff_pass}, y={primal}] {progress_5_150_Adaptive_perm_20.txt};
      \addplot+[thick, mark=oplus, color=cb-rose] table[x={eff_pass}, y={primal}] {progress_10_150_Adaptive_perm_20.txt};
      \addplot+[thick, mark=triangle, mark repeat=15, color=cb-burgundy] table[x={eff_pass}, y={primal}] {progress_1_150_Optimal_perm_20.txt};
      \addplot+[thick, mark=pentagon, color=cb-green-sea] table[x={eff_pass}, y={primal}] {progress_1_150_Optimal_uniform_20.txt};
    \end{axis}
  \end{tikzpicture}
}
\subfloat{
  \centering
  \begin{tikzpicture}[scale=1.0]
    \begin{axis}[height=4.2cm,width=4.1cm, legend cell align={left},
      ylabel style={yshift=-0.75em,font=\scriptsize}, 
      xlabel style={yshift=0.50em,font=\scriptsize},
      legend entries={1 Adaptive P-Cyclic, 5 Adaptive P-Cyclic, 10 Adaptive P-Cyclic, 1 Linesearch P-Cyclic, 1 Linesearch Uniform},
      legend style={font=\tiny},
      tick label style={font=\tiny}, 
      legend style={font=\tiny, at={(1.0,2.0)}, anchor=north east, legend columns=3},
      legend to name={legend_bcfwstruct},
      scaled x ticks=false,
      xlabel=Time (sec),
      grid,
      xmin=0, xmax=35, ymin=0.15, ymax=1, ymode=log, mark repeat=50]

      % actual plots
      \addplot+[thick, mark=square, color=cb-blue] table[x={time}, y={primal}] {progress_1_150_Adaptive_perm_20.txt};
      \addplot+[thick, mark=10-pointed star, color=cb-clay] table[x={time}, y={primal}] {progress_5_150_Adaptive_perm_20.txt};
      \addplot+[thick, mark=oplus, color=cb-rose] table[x={time}, y={primal}] {progress_10_150_Adaptive_perm_20.txt};
      \addplot+[thick, mark=triangle, mark repeat=15, color=cb-burgundy] table[x={time}, y={primal}] {progress_1_150_Optimal_perm_20.txt};
      \addplot+[thick, mark=pentagon, color=cb-green-sea] table[x={time}, y={primal}] {progress_1_150_Optimal_uniform_20.txt};
    \end{axis}
  \end{tikzpicture}
}
\end{center}
\caption{Results of the structural SVM training experiment (Section~\ref{sec:ex3}) displaying the primal progress versus iterations, epochs and time.
comparing Algorithm~\ref{a:abcg} using P-Cyclic updates of size
$1$, $5$, and $10$, and the original method in \cite{Laco13} using
an exact line search with P-Cyclic and uniformly random singleton updates.}
  \label{fig:ex3}
\end{figure}

\section{Future work}

Here we conclude by outlining potential directions for future work.
This article focuses on the short-step family of step sizes, and
there is not a clear avenue for generalizing to
other step size families, e.g., open-loop or line search
\cite[Algorithm~3]{Laco13}. Such a task may be more tractable on
separable objective functions (note neither objective in 
Examples~\ref{ex:nonconv},\ref{ex:nonconv2} are
separable). The present analysis also does not generalize to the
case of nonsmooth or partially nonsmooth functions, which can be
common in applications.
\begin{acknowledgements}
This research was partially supported by the DFG Cluster of
Excellence MATH+ (EXC-2046/1, project id 390685689) funded by the
Deutsche Forschungsgemeinschaft (DFG) and took place on the Research
Campus MODAL funded by the German Federal Ministry of Education and
Research (BMBF) (fund numbers 05M14ZAM, 05M20ZBM). The work of ZW
is also supported by NSF DMS-2532423. Part of this
work was conducted while SP was visiting Tokyo University via a
JSPS International Research Fellowship.

We thank our referees, as well as D. Hendrych, G. Iommazzo,
D. Kuzinowicz, M. Turner, and E. Wirth for valuable
feedback on early drafts of this article.\newline 
\noindent
\textbf{Data Availability Statement}\;\;
 The code used to produce these results can
be found at \url{https://github.com/zevwoodstock/BlockFW} and
\url{https://github.com/JannisHal/BCFWstruct}. New
\texttt{LazyUpdate} and
\texttt{adaptive\_block\_coordinate\_frank\_wolfe} methods are also
in\newline
\url{https://github.com/ZIB-IOL/FrankWolfe.jl}.
\end{acknowledgements}
{\footnotesize
\bibliography{\bibfilename}
}
\end{document}